\documentclass[a4paper,11pt,reqno]{amsart}

\usepackage[USenglish]{babel}
\usepackage[T1]{fontenc}
\usepackage[utf8]{inputenc}
\usepackage{amsfonts}

\usepackage{amsmath}

\usepackage{amssymb}

\usepackage{newtxtext}

\usepackage{amsthm}

\usepackage{xcolor}

\usepackage{fix-cm}

\usepackage[colorlinks]{hyperref}
\hypersetup{
  linkcolor=[RGB]{192,0,0},
  citecolor=[RGB]{0, 150, 0},
}

\usepackage{braket}
\usepackage{mathtools}
\usepackage{thmtools}
\usepackage{graphicx}
\usepackage{tabularx}
\usepackage{tensor}

\usepackage{tikz}
\usepackage{tikz-cd}
\usepackage{tikz-3dplot}
\usetikzlibrary{cd}
\usetikzlibrary{arrows}
\usetikzlibrary{matrix}
\usepackage{nicematrix} 
\usetikzlibrary{positioning}

\tikzcdset{arrow style=math font}

\usepackage[subscriptcorrection]{newtxmath}

\DeclareMathAlphabet{\mathbb}{U}{msb}{m}{n}
\DeclareMathAlphabet{\mathfrak}{U}{euf}{m}{n}

\DeclareMathAlphabet{\mathbb}{U}{msb}{m}{n}

\usepackage{xcolor}
\definecolor{red}{rgb}{1,0,0}
\definecolor{darkred}{RGB}{192,0,0}

\newcommand{\Sum}{\displaystyle\sum}

\newcommand{\bigzero}{\mbox{\huge 0}}

\newcommand{\bfa}{\mathbf{a}}

\newcommand{\bfk}{\mathbf{k}}

\newcommand{\bfm}{\mathbf{m}}

\newcommand{\bfx}{\mathbf{x}}
\newcommand{\bfy}{\mathbf{y}}

\newcommand{\calA}{\mathcal{A}}
\newcommand{\calB}{\mathcal{B}}

\newcommand{\calD}{\mathcal{D}}

\newcommand{\calF}{\mathcal{F}}

\newcommand{\calP}{\mathcal{P}}

\newcommand{\calR}{\mathcal{R}}

\newcommand{\calZ}{\mathcal{Z}}

\newcommand{\bbA}{\mathbb{A}}

\newcommand{\bbC}{\mathbb{C}}

\newcommand{\bbK}{\mathbb{K}}

\newcommand{\bbN}{\mathbb{N}}

\newcommand{\bbP}{\mathbb{P}}

\newcommand{\bbZ}{\mathbb{Z}}

\newcommand{\bfalpha}{\boldsymbol{\alpha}}

\newcommand{\dashto}{\dashrightarrow}

\renewcommand{\tilde}[1]{\widetilde{#1}}
\renewcommand{\hat}[1]{\widehat{#1}}

\newcommand{\id}{\mathrm{id}}

\DeclareMathOperator{\Ann}{Ann}
\DeclareMathOperator{\Ap}{Ap}

\DeclareMathOperator{\brk}{brk}
\DeclareMathOperator{\Cat}{Cat}

\DeclareMathOperator{\HF}{HF}

\DeclareMathOperator{\im}{im}

\DeclareMathOperator{\ldf}{ldf}

\DeclareMathOperator{\Proj}{Proj}

\DeclareMathOperator{\rk}{rk}

\DeclareMathOperator{\smrk}{smrk}
\DeclareMathOperator{\crk}{crk}

\DeclareMathOperator{\Spec}{Spec}

\DeclareMathOperator{\tdf}{tdf}

\newcommand{\Hilb}{\mathrm{Hilb}}

\usepackage[textwidth=16cm, textheight=24cm]{geometry}

\usepackage{enumitem}

\usepackage[alphabetic]{amsrefs}

\usepackage[capitalise,nameinlink]{cleveref}

\usepackage{adjustbox}
\usepackage[textwidth=0.8in]{todonotes}
\usepackage{comment}
\definecolor{darkred}{RGB}{192,0,0}
\definecolor{darkelectricblue}{rgb}{0.33, 0.41, 0.47}

\newcommand{\oldtext}[1]{\textcolor{darkelectricblue}{#1}}
\newcommand{\cfchange}[1]{\textcolor{blue}{#1}}

\newcommand{\oldinvtext}[1]{}

\allowdisplaybreaks

\numberwithin{equation}{section}
\theoremstyle{definition}
\newtheorem{defn}[equation]{Definition}
\theoremstyle{plain}
\newtheorem{teo}[defn]{Theorem}
\newtheorem{prop}[defn]{Proposition}

\newtheorem{lem}[defn]{Lemma}

\newtheorem{cor}[defn]{Corollary}
\theoremstyle{remark}
\newtheorem{rem}[defn]{Remark}
\newtheorem{exam}[defn]{Example}
\newtheorem{question}[defn]{Question}
\allowdisplaybreaks
\def\CC{\mathbb{C}}

\title{Symmetric powers: structure, smoothability, and applications}
\author{Cosimo Flavi, Joachim Jelisiejew, Mateusz Micha{\l}ek}
\address{{\normalfont (Cosimo Flavi)}
\normalfont \scshape\fontfamily{ptm}\selectfont
	Wydział Matematyki, Informatyki i Mechaniki,
Uniwersytet Warszawski,
	\normalfont{ul.~Stefana Banacha 2, 02-097 Warsaw, Poland.}}
\email{c.flavi@uw.edu.pl}

\address{{\normalfont (Joachim Jelisiejew)} \normalfont \scshape\fontfamily{ptm}\selectfont
	Wydział Matematyki, Informatyki i Mechaniki,
Uniwersytet Warszawski,
	\normalfont{ul.~Stefana Banacha 2, 02-097 Warsaw, Poland.}}
\email{j.jelisiejew@uw.edu.pl}

\address{{\normalfont (Mateusz Micha{\l}ek)} \normalfont \scshape\fontfamily{ptm}\selectfont
	Fachbereich Mathematik und Statistik,
Universit\"at Konstanz,
	\normalfont{Fach D 197 D-78457, Konstanz, Germany.}}
\email{mateusz.michalek@uni-konstanz.de}



\DeclareMathOperator{\grad}{grad}
\DeclareMathOperator{\totalgrad}{totalGrad}
\newcommand{\goodquotient}{\mathbin{
  \mathchoice{\left/\mkern-6mu\right/}
    {/\mkern-5mu/}
    {/\mkern-5mu/}
    {/\mkern-5mu/}}}

\newcommand{\calDgl}{\calD^{\rm{gl}}}
\newcommand{\calRgl}{\calR^{\rm{gl}}}

\crefname{lem}{lemma}{lemmas}
\Crefname{examp}{Example}{Examples}

\begin{document}
\renewcommand{\abstractname}{\normalfont \scshape\fontfamily{ptm}\selectfont{Abstract}}
\begin{abstract}
    We investigate border ranks of twisted powers of polynomials and
    smoothability of symmetric powers of algebras. We prove that the latter
are smoothable. For the former, we obtain upper bounds for the border rank in
general and prove that they are optimal under mild conditions. We give
applications to complexity theory. {Many of the results rest on the notion of an \emph{encompassing polynomial}, which we introduce.}
\end{abstract}
\maketitle

\thispagestyle{empty}

\section{\normalfont \scshape\fontfamily{ptm}\selectfont Introduction}
    \oldinvtext{We consider two classical problems: finding
    smoothable finite algebras and finding
    homogeneous polynomials of small border rank.
    \oldtext{Both problems are understood to be next-to-impossible in general and} \cfchange{Both problems are generally considered to be almost impossible, and}
    positive results are \oldtext{very} scarce except for specific cases, such as
    subschemes of $\mathbb{A}^2$~\cite{fogarty}, Gorenstein subschemes of
    $\mathbb{A}^3$~\cite{kleppe_roig_codimensionthreeGorenstein}, forms in
    small number of variables~\cite{Landsberg__tensors}, subschemes defined by monomial ideals or complete intersections.
    Outside this realm, the main general tool to obtain either object is using tensor powers:}

{We study two classical and challenging problems in algebraic geometry and commutative algebra: the classification of smoothable finite algebras, and the classification of homogeneous polynomials of small border rank. Both problems are widely regarded as very difficult, primarily due to the vast diversity of possible examples and lack of general structure theorems.}

{Nevertheless, significant progress has been made in particular cases. Notable examples include subschemes of $\mathbb{A}^2$~\cite{fogarty}, Gorenstein subschemes of $\mathbb{A}^3$~\cite{kleppe_roig_codimensionthreeGorenstein}, and homogeneous forms in a small number of variables~\cite{Landsberg__tensors}. Additional constructions involve subschemes defined by monomial ideals or complete intersections, as well as explicit constructions derived via computational methods such as limits of families of points and linear combinations of powers of forms.}

{Outside these specific instances, one of the few general techniques applicable to both smoothability and border rank involves the use of tensor powers. This method yields the following two key properties:}
    \begin{enumerate}[label=(\arabic*), left= 3pt, widest=2,nosep]
        \item if $A$ is a smoothable algebra, then for every $d$ the tensor
            product $A^{\otimes d}$ is smoothable;
        \item\label{it:tensor} if $F = F(x_0, \ldots ,x_n)$ is a homogeneous polynomial of border rank at most
            $r$, then for every $d$ the polynomial
            \[
                F^{\otimes d} := F(x_{1,0}, \ldots ,x_{1,n})
                F(x_{2,0}, \ldots ,x_{2,n})\cdots F( x_{d,0},
                \ldots ,x_{d,n})
            \]
            in $(n+1)d$ variables has border rank with respect to the
            Segre-Veronese variety (=variety of monomials of the same
                multi-degree as $F^{\otimes d}$) at most
            $r^{d}$.
    \end{enumerate}

    In this article we provide twisted \emph{symmetric power} analogues of the two results, see
    \autoref{ref:intro:smoothability} and
    \autoref{ref:intro:universalityTheorem} below. Both results are special
    cases of more general structure {theorems, which we} explain below.
    The passage from tensor to symmetric powers may seem to be automatic, but
    it is not, in fact it is unexpectedly subtle,
    {see Counterexamples~\ref{exm:VerNotSmooth},~\ref{ex:counterexampleBR},~\ref{ex:bigSquareForEncompassing}}
    below and the discussion about the necessity of twisting. In fact, while
    the tensor powers are a classical tool, it seems that almost nothing is
    known about symmetric powers.

    Smoothability is of {great} interest for moduli spaces, with applications
    to enumerative geometry and combinatorics. Border rank is central {to the} classical theory of secant varieties (see~\cites{BCC+18, BGI11,CGO14}), with
    applications to statistics, signal processing (see~\cites{Che11, DC07,
    McC87}) and, {expecially}, to complexity theory and the
    geometry of tensors (\cites{BCS97, landsberg2017geometry, CW}). From this point of view,
    it is important to look at the \emph{asymptotical} behavior: to consider
    sequences of polynomials or algebras with {increasing} degree and number of
    variables. In this aspect, the gain from passing from tensor to symmetric
    powers is very notable. Also in the symmetric power setting the
        rank we consider is always the usual Waring rank, in contrast {to} the
        right hand side of~\ref{it:tensor}
above, where it changes to the Segre-Veronese rank.

\subsection{Smoothability}\label{ssec:smoothability}

    The situation is much cleaner for smoothability, so we begin with {this part}.
    We talk about finite schemes $Z = \Spec(A)$ rather than algebras $A$.
    By
    $\Ap(-)$ we denote the apolar algebra.

\begin{teo}\label{ref:intro:smoothability}
    \begin{enumerate}[label=(\arabic*), left= 3pt, widest=2,nosep]
        \item For every {finite} smoothable scheme $Z$, its $d$-th symmetric power $S^d Z$ is smoothable for every {$d\geq 1$};
        \item if $f\in \bbK[x_1, \ldots ,x_n]$ and $d\geq 1$ are such that \[
        \dim_{\bbK}
            \Ap(f^d) = \binom{\dim_{\bbK} \Ap(f)+d-1}{d},
            \]
            then the scheme \[Z_d = \Spec \bigl(\Ap(f^d)\bigr)\] is isomorphic to $S^dZ_1$. {In particular,} if $Z_1$ is smoothable, then $Z_d$ is smoothable.
    \end{enumerate}
\end{teo}
{Note} that in general the smoothability of $\Spec \bigl(\Ap(f)\bigr)$ does not imply the
smoothability of $\Spec \bigl(\Ap(f^d)\bigr)$ for $d\geq 1$. In
\autoref{ex:squareNotSmoothable} we provide an $f$ such that $\Ap(f)$ is
smoothable, while $\Ap(f^2)$ is not.
The proof of \autoref{ref:intro:smoothability} is based on the following
abstract but useful proposition.
\begin{prop}\label{ref:quotients:prop}
    Let $Z$ be a finite scheme with an action of a linearly reductive algebraic group $G$ and
    suppose that $Z$ admits a $G$-equivariant smoothing $\mathcal{Z}$. Then
    $\mathcal{Z}\goodquotient G$ is a smoothing of $Z\goodquotient G$, so
    $Z\goodquotient G$ is smoothable.
\end{prop}
For the case $\Spec \bigl(\Ap(f^d)\bigr)$, we have $Z = \Spec \bigl(\Ap(f)^{\otimes d}\bigr)$ and $G =
\Sigma_d$
is the symmetric group acting on $Z$ by permuting the copies of $\Ap(f)$.
\autoref{ref:quotients:prop} applies in other natural situations, notably for
the Veronese subalgebras. We discuss {these} now.
\begin{exam}
    Fix an arbitrary $d$ and consider the monomial algebra
    \[
        {A(d)} \coloneqq \frac{\bbK[\alpha_1, \ldots ,\alpha_{d}]}{(\alpha_1^2,  \ldots
            ,\alpha_{d}^2)}
    \]
    and the corresponding scheme $Z_d \coloneqq \Spec({A(d)})$.
    Let $G \coloneqq \langle g\rangle$ be a cyclic group of order two that acts on
    \[Z_1 = \Spec\bigl(\bbK[\alpha]/(\alpha^2)\bigr)\] by $g\cdot \alpha \coloneqq -\alpha$.
    The
    scheme $Z_1$ has a $G$-equivariant smoothing $\bbK[\alpha, t]/(\alpha^2 -
    t)$. Therefore, also the scheme $Z_d = Z_1\times \cdots \times Z_1$ admits
    a $G$-equivariant smoothing, so
    \[
        Z_d\goodquotient G  \simeq \Spec\bigl( {A(d)}^{\mathbb{Z}/2} \bigr)
    \]
    is smoothable.
\end{exam}

\autoref{ref:quotients:prop} is sharp in that the
assumption of the existence of $G$-equivariant smoothing is
necessary, as we {show} in the following example.

\begin{exam}\label[examp]{exm:VerNotSmooth}
Let us consider the monomial algebra
    \[
        A \coloneqq \frac{\bbK[\alpha_1, \ldots ,\alpha_9]}{(\alpha_1^2,  \ldots
        ,\alpha_9^2)}
    \]
    which has degree $512$ and Hilbert function \[\HF_A = ( 1, 9, 36, 84,
    126, 126, 84, 36, 9, 1).\]
    The corresponding subscheme $Z \coloneqq \Spec(A) \subseteq
    \mathbb{A}^9$ is torus-fixed and in particular the cyclic group $G \coloneqq
    \mathbb{Z}/3$ acts on it by multiplying each variable by a third root of unity. In this case
    \[
        Z\goodquotient G  \simeq \Spec\bigl(A^{\mathbb{Z}/3}\bigr) = \Spec
        \left(\bbK \oplus A_3 \oplus A_6 \oplus A_9\right).
    \]
    This is a subscheme with Hilbert function $(1,84,84,1)$, which
    is not smoothable by the criterion of~\cite{Jel19}*{Theorem~1.2}.
\end{exam}

Let us define a more general class of algebras which includes the last two
examples.
\begin{defn}
Given a graded algebra $A=\bigoplus_{i\geq 0} A_i$ and a positive integer $k$
the algebra $\bigoplus_{i\geq 0} A_{ki}$ is called the \emph{$k$-th Veronese subalgebra}.
\end{defn}

The previous examples, as well as our theorems on passing to symmetric powers,
can be seen as specific cases of a more general question:
\begin{question}
Given a finite smoothable scheme $\Spec(A)$ with an action of a group $G$,
when is $\Spec(A^G)$ also smoothable?
\end{question}
As we can see in \autoref{exm:VerNotSmooth} the answer is not always positive. Still, \autoref{ref:intro:smoothability} and \autoref{ref:quotients:prop} provide many examples when it is.

\subsection{Powers of polynomials and their ranks}

We now shift {our} attention from smoothability to border rank.
    Let $F\in \bbK[x_0, \ldots ,x_n]$ be a homogeneous polynomial.
    {One can mimic~\ref{it:tensor} above as follows: replace $r^d = \dim_{\bbK} (\bbK^r)^{\otimes d}$ with $\binom{r+d-1}{d} = \dim_{\bbK} S^d(\bbK^r)$ and ask whether} if $\brk(F) {\leq} r$, then \[\brk(F^d) \leq
    \binom{r+d-1}{d}\] for all $d\geq 1$. {However, this inequality} is false even in the {simplest} cases.

    \begin{exam}\label[examp]{ex:counterexampleBR}
        Let \[F \coloneqq x_0^3 + x_1^3 \in \bbK[x_0, x_1].\] Then $\rk(F) = 2$ and
        the middle catalecticant rank of $F^2$ is $4 > \binom{3}{2}$, so
        $\brk(F^2) \geq 4$, in fact $\brk(F^2) = 4$. Also $\brk(F^{[2]}) = 4$,
        where $F^{[2]}$ is the divided square of $F$. Similar results will be
        obtained for Fermat cubics in more variables, for exponents {high}
        than $3$ and so on.
    \end{exam}

    Since the upper bound on the border rank is {false} in the {simplest} case of
    two variables, it may come as a surprise that a bound can be achieved
    in a fairly general situation, which we explain below. We need two
    notions.

    \newcommand{\twist}[1]{\mathrm{tw}(#1)}

    First, for a homogeneous
    polynomial $F = \sum_{i}
    \lambda_i \bfx^{i}\in \bbK[x_0, \ldots ,x_n]$, where $i=(i_0,\dots,i_n)$, we define the \emph{twist of $F$} by
    \[
        \twist{F} := \sum_{i} \lambda_i\frac{1}{i_0!} \bfx^i.
    \]
    For example,
    \[\twist{x_0^2x_1x_2 + x_0^3 x_2 + x_1^4} = \frac{1}{2}x_0^2x_1x_2 +
    \frac{1}{6}x_0^3x_2 + x_1^4.\]
    The reparametrization used in the twist is unfortunately necessary below, as examples such as \autoref{ex:bigSquareForEncompassing} show. The technical reason for the need of twists is explained in~\S\ref{ssec:tautological} below. It would be very interesting to find a more conceptual explanation.

Second, we say that
a polynomial $f$ is \emph{encompassing} if the parts of degree $\leq 1$ of
a
basis of the space of its partials are linearly independent, see~\autoref{defn:encompassing}.
For example, the polynomial $f\coloneq x_1^2 + x_2$ is
encompassing, because the degree $\leq 1$ part of derivatives of $f$ are
\[(f)_{\leq 1}=x_2,\qquad \frac{\partial
f}{\partial x_1} =\biggl(\frac{\partial
f}{\partial x_1}\biggr)_{\leq 1} =2x_1,\qquad\frac{\partial^2 f}{\partial
x_1^2} = \biggl(\frac{\partial^2 f}{\partial x_1^2}\biggr)_{\leq 1}=2,\] and
they are
linearly independent. The form $F \coloneq x_1^2 +x_2^2$ is, instead, not
encompassing, because already
$F_{\leq 1}$ is zero {(for this reason, homogeneous polynomials of degree $\geq 2$ are never encompassing)}.
Our second theorem is as follows.
\begin{teo}\label{ref:intro:mainThmForms}
    Let $F\in \bbK[x_0, \ldots ,x_n]$ be a concise homogeneous polynomial such that
    its dehomogenization $f\coloneq F|_{x_0=1}$ is encompassing and the apolar
    algebra of $f$ is smoothable. Then, for every $d\geq 1$, the following
    conditions hold:
\begin{enumerate}[label=(\arabic*), left= 3pt, widest=2,nosep]
        \item the form $\twist{F^d}$ has smoothable rank, border
            rank, cactus
            rank, border cactus rank, and
            the middle catalecticant rank all equal to $\binom{n+d}{d}$;
        \item the scheme
            $\Spec \bigl(\Ap(f^d)\bigr)$ is smoothable and apolar to
            $\twist{F^d}$ under a
            suitable embedding (see \S\ref{ssec:tautological}).
    \end{enumerate}
\end{teo}
\begin{exam}\label{ex:quadrics}
    Let $Q$ be a full rank quadric given by \[Q \coloneq x_0x_n + x_1^2 + x_2^2 +  \ldots +
    x_{n-1}^{2}.
\]
    The dehomogenization $p_n \coloneq Q|_{x_0 = 1}$ is encompassing.
    By \autoref{ref:intro:mainThmForms}, for every $d\geq 1$ the form
    $\twist{Q^d}$ has
    smoothable and border ranks equal to
    \[
        \binom{n+d}{d}.
    \]
    For ternary
    quadratic forms and without twisting this result was obtained by the first author
    in~\cite{Fla23}*{Theorem 4.5}. We warn that not every dehomogenization
    of $Q$ is encompassing: consider $Q|_{\sum_{i=0}^n \lambda_ix_i = 1}$, where $\lambda_i\in \bbK$. If $Q(\lambda_0,\ldots,\lambda_n) = 0$, then, up to coordinate change, $Q|_{\sum_{i=0}^n \lambda_ix_i = 1}$ is $Q|_{x_0 = 1}$, so it is encompassing. However, if $Q(\lambda_0,\ldots,\lambda_n) \neq 0$, then $Q|_{\sum_{i=0}^n \lambda_ix_i = 1}$ is, up to coordinate change, equal to $1 + x_1^2 + \ldots + x_n^2$, which is not encompassing. In particular, a general dehomogenization is not encompassing.
\end{exam}
\begin{exam}\label{ex:cubics}
    Let us consider the cubic
    \[F \coloneq x_1^3 + x_2^3 +  \dots + x_n^3 + x_{0}\left( x_1y_1 +  \dots +
    x_ny_n \right) + x_0^2y_0\] in $2(n+1)$ variables. The
    dehomogenization $F|_{x_0 = 1}$ is encompassing and the apolar algebra
    \[\Ap(F|_{x_0=1})  \simeq \Ap(x_1^3 + x_2^3 +  \ldots + x_n^3)\] is
    smoothable, see~\cite[Theorem~3.3]{EliasRossiShortGorenstein} or~\cite[Example~2.16]{Jel17} for the isomorphism. Hence, by
    \autoref{ref:intro:mainThmForms}, for every $d\geq 1$ we have
\[
    \brk\bigl(\twist{F^d}\bigr)=\binom{2n+d+1}{d}.
\]

\end{exam}
We warn the reader that twisting is necessary: even for $F$ satisfying the
assumptions of \autoref{ref:intro:mainThmForms}, the border rank of {both the square} $F^2$ and {the divided square}
$F^{[2]}$ can be higher than $\binom{n+1}{2}$, see
\autoref{ex:bigSquareForEncompassing}.

The theorem implies that $\twist{F}$ has the minimal possible border
rank $n$ and smoothable rank also equal to $n$, so $\twist{F}$ is not wild in the
sense of~\cite{BB15}.
A natural question is: \emph{how limiting are the assumptions in
    \autoref{ref:intro:mainThmForms}?} The following theorem shows that the
    smoothability of $\Ap(f)$ is the primary requirement: {we may always construct an encompassing polynomial that restricts to ours by adding a \emph{minimal possible} number of new variables.}

We say that a polynomial $G\in \bbK[x_0, \ldots ,x_n]$ \emph{restricts to} a
polynomial $F\in \bbK[x_0, \ldots ,x_k]$, where $k\leq n$, if $F =
G|_{x_{k+1} =0, \ldots ,x_n=0}$.

\begin{teo}\label{ref:intro:universalityTheorem}
    Let $F\in \bbK[x_0, \ldots ,x_k]$ be a concise homogeneous polynomial, let
    $f\coloneq F|_{x_0=1}$ be its dehomogenization, and assume that $\Ap(f)$
    be smoothable. If $\dim_{\bbK} \bigl(\Ap(f)\bigr)=n$, then there exists a concise
    homogeneous polynomial $G\in \bbK[x_0, \ldots ,x_n]$ restricting to
    $F$ and such that $G$ satisfies the assumptions of
    \autoref{ref:intro:mainThmForms}. In particular,
    \[
        \brk(\twist{F^d})\leq\binom{n+d}{d},
    \]
    for every $d\in\bbN$.
    Moreover, for $g\coloneq G|_{x_0=1}$ we have $\Ap(g) \simeq
    \Ap(f)$.
\end{teo}
The proof of \autoref{ref:intro:universalityTheorem} is constructive and the form $G$
is determined explicitly. Namely, let us
take a concise homogeneous $F\in \bbK[x_0,  \ldots ,x_k]$ and $f = F|_{x_0=1}$. The apolar
algebra $\Ap(f)$ is a quotient of the dual ring $\calD := \bbK[\alpha_1, \ldots
,\alpha_k]$, where
\[\alpha_i\coloneq\frac{\partial}{\partial x_i}\]
for every $i=1,\dots,k$.
Let us consider a basis $\calB$ of $\Ap(f)$ given by
\[
\calB\coloneq\{ f, \alpha_1\circ f, \ldots, \alpha_k \circ f, \sigma_{1}\circ f, \ldots ,\sigma_{n-k-1}\circ f \},
\]
where $\sigma_i\in \calD_{\geq 2}$.
For any multi-index $\bfa\in \mathbb{Z}_{\geq 0}^{n-k-1}$, let us define
\[\sigma^{\bfa} \coloneq \sigma_1^{\bfa_1} \cdots
    \sigma_{n-k-1}^{\bfa_{n-k-1}}\quad\mbox{and}\quad \bfy^{\bfa} \coloneq y_1^{\bfa_1} \cdots
y_{n-k-1}^{\bfa_{n-k-1}}.\]
Then the polynomial $g\in \bbK[x_1, \ldots ,x_k, y_{1}, \ldots ,y_{n-k-1}]$ is defined by a
Taylor-like series
\begin{equation}\label{eq:exponent}
        g\coloneq \sum_{\bfa\in \mathbb{Z}_{\geq 0}^{n-k-1}} \frac{\bfy^{\bfa}}{\bfa!}(\sigma^{\bfa}\circ f)
\end{equation}
and $G\in \bbK[x_0, \ldots ,x_k, y_1, \ldots ,y_{n-k-1}]$ is the
homogenization of $g$ with respect to $x_0$. To obtain $G$ in $\bbK[x_0, \ldots
,x_n]$ one takes $x_j \coloneq y_{j-k}$ for $j > k$.
\begin{exam}
    To obtain \autoref{ex:quadrics}, let us consider the form
    \[F \coloneq x_1^2 +  \cdots + x_{n-1}^2\in
    \bbK[x_0, \ldots ,x_{n-1}].
    \]
    Then $f \coloneq F|_{x_0=1}$ is essentially equal to
    $F$. Let us fix the set
    \[
    \{1, \alpha_1, \ldots,\alpha_{n-1}, \sigma_1\},
    \]
    where $\sigma_1 \coloneq \frac{1}{2}\alpha_1^2$,
    as the basis of $\Ap(f)$. Then, we have \[
    \sigma_1^2\circ f = \sigma_1 \circ 1 = 0,\]
    so that the polynomial defined in formula \eqref{eq:exponent} is $g=f + y_1$. The homogenization with respect to
    $x_0$ yields \[ x_1^2 +  \cdots + x_{n-1}^2 + x_0y_1,\] as expected.
\end{exam}
\begin{exam}
    To obtain \autoref{ex:cubics}, let us consider the form
    \[ F \coloneq x_1^3 + \cdots + x_n^3\in
    \bbK[x_0, \ldots ,x_n].\] Then, we take $f \coloneq F|_{x_0=1}$ and
    as a basis of $\Ap(f)$ fix the set
\[
\{1,\alpha_1,\dots,\alpha_n,\sigma_1,\dots,\sigma_n,\sigma_{n+1}\},
\]
with $\sigma_i \coloneq \alpha_i^2/2$ for $i=1, \ldots , n$ and
    $\sigma_{n+1} \coloneq \alpha_1^3/6$.
    For every $i,j$ we have $\sigma_i\circ (\sigma_j \circ f) = 0$, so that
    the formula \eqref{eq:exponent} becomes
    \[
      g =  f + x_1y_1 +  \ldots + x_ny_n + y_{n+1},
    \]
    which homogenizes to the form from the example.
\end{exam}

Having discussed the abundance of encompassing polynomials, we present a more geometric point of view for them.
In the projective setting, there exist elegant classical links between the algebra of
a homogeneous polynomial $F\in \bbK[x_0, \ldots ,x_n]$ and the topology of its
gradient map
\[
    \grad(F) = \left( \frac{\partial F}{\partial x_0},  \ldots ,
    \frac{\partial F}{\partial x_n} \right) \colon \mathbb{P}^{n}\dashto \mathbb{P}^{n},
\]
also called the \textit{polar map}. For further details, the reader can see, e.g., \cites{huh2012milnor, DP03, Dol00,
HKS92}.

In the affine setting, we prove the following analogue which in
particular gives a characterization of polynomials which have dominant affine
gradient maps and a characterization of polynomials with maximal growth of
powers. We need one definition. By \autoref{prop:inequalityOnDuals} for every
polynomial $f$ with $\dim_{\bbK}\bigl(\Ap(f)\bigr)=\ell$ and every $d\geq 1$
we have
\begin{equation}\label{eq:growth}
\dim_{\bbK} \Ap(f^d) \leq \binom{\ell + d-1}{d}.
\end{equation}
We say that $f$ has \textit{maximal growth of powers} if for every $d\geq 1$ equality holds
in~\eqref{eq:growth}.
\begin{teo}\label{ref:intro:growth}
    Let $f\in \bbK[x_1, \ldots ,x_n]$ be a concise polynomial {and take $\ell \coloneq \dim_{\bbK}\bigl(\Ap(f)\bigr)$}. Then, the following conditions
    are equivalent:
    \begin{enumerate}[label=(\arabic*), left= 3pt, widest=2,nosep]
        \item\label{it:growthOne} $f$ has maximal growth of powers;
        \item\label{it:growthTwo} the \emph{total gradient} map
            \[
                \totalgrad(f) = \left( \frac{\partial f}{\partial \bfx^{\bfa_1}}, \ldots ,
                \frac{\partial f}{\partial \bfx^{\bfa_{\ell-1}}} \right)\colon
                \mathbb{A}^{n}\dashto
                \mathbb{A}^{\ell-1}
            \]
            is dominant, where the set
            \[ \biggl\{\frac{\partial f}{\partial
                \bfx^{\bfa_{i}}}\biggr\}_i\cup\{1\}\] forms a basis of the space of all
                partials of $f$;
        \item\label{it:growthThree} the partials of $f$ are homogeneously
            algebraically independent (see \S\ref{sec:encompassing});
        \item\label{it:growthFour} $f$ is an encompassing polynomial.
    \end{enumerate}
    Moreover, if these conditions hold, then $\ell = n+1$ and the {total} gradient map $\totalgrad(f)$
    is, up to coordinate changes, the usual gradient map
    \[
        \grad(f) = \left( \frac{\partial f}{\partial x_1}, \ldots ,
    \frac{\partial f}{\partial x_n} \right)\colon
    \mathbb{A}^n\dashto
    \mathbb{A}^n.
    \]
\end{teo}
A word of caution is that the rational map $\grad(f)$ yields a
rational map $\mathbb{P}^n\dashto \mathbb{P}^n$, but this map is in
general \emph{not} equal to $\grad(f^h)$ for a homogenization
$f^h$ of $f$.

Outside the class of encompassing polynomials, the following natural question
seems absolutely open.
\begin{question}
    Which sequences can be obtained as $\bigl(\dim_{\bbK} \Ap(f^d)\bigr)_{d\geq 1}$?
\end{question}
Specializing a bit, one could wonder
whether knowing first few dimensions is enough to know the whole sequence.
This is an open question as well. In \autoref{ref:growthuptodegImpliesEncompassing:cor} we answer it for
the sequence \[d\mapsto \binom{\ell+d-1}{d}.\]
We show that
if $\dim_{\bbK} \Ap(f^d)$ is maximal for all $d\leq \deg(f)$, then $f$ is
encompassing.

\subsection{Applications to complexity theory} In the last part of the article our main results are:
\begin{enumerate}[label=(\arabic*), left= 3pt, widest=2,nosep]
    \item \autoref{cor:smoothabilitybrank1}: deciding {the} smoothability of
        algebras implies estimating border rank of tensors, up to a
    factor of $1/2$,
\item \autoref{cor:brankto1Gen}: deciding {the} smoothability of modules implies estimating border rank of tensors $T\in
    \bbK^n\otimes\bbK^n\otimes \bbK^n$ up to an error of $n$,
\item introduction of \emph{sweet pieces}, that is a special class of
    tensors appearing in complexity theory, see~\autoref{def:sweet},
\item estimates of {the} ranks of sweet pieces of interest in \autoref{prop:sweetRank} and \autoref{prop:Prattimproved}.
\end{enumerate}
One of our motivations comes from applications to complexity theory. More precisely, the two major challenges relevant for us are:
\begin{enumerate}[label=(\arabic*), left= 3pt, widest=2,nosep]
\item fast matrix multiplication;
\item special NP-hard problems, like chromatic number of a graph or the set cover problem.
\end{enumerate}
It turns out that in both of the above problems good upper bounds on ranks of Kronecker powers of special tensors can lead to very surprising upper bounds on complexity.
Let us explain in detail how smoothable algebras, their tensor powers, and Veronese subalgebras appear in {the above problems.}

The complexity of matrix multiplication is governed by the constant $\omega$, see \cites{BCS97, landsberg2017geometry} for a detailed exposition, or \cite{michalek2021invitation}*{Chapter 9.3} for a {quick} introduction to the topic.
 The {state-of-the-art} method to bound $\omega$ from above is due to
 D.~Coppersmith and S.~Winograd (see~\cite{CW}). The exact value of $\omega$ remains
 unknown, although upper bounds have been improved in recent years, showing
 that $\omega<2.38$, see~\cites{CW, alman2021refined, williams2023new, alman2024more}. The
 core method starts with a smoothable algebra $A$, which is apolar to a quadric.
 The bilinear map given by multiplication in $A$ gives rise to a tensor $T_A$,
 known as the Coppersmith-Winograd tensor,{which is} of minimal border rank
 $\dim_{\bbK} A$, see~\cite{BL16}. One identifies direct sums of
 matrix multiplication tensors as restriction of the tensor associated to
 $A^{\otimes d}$ for large $d$, which is also the $d$-th Kronecker power of
 $T_A$. This allows to upper bound the border rank of the direct sums by
 $(\dim_{\bbK}A)^d$. {Using} Sch\"onhage's $\tau$-theorem~\cite{BCS97}*{\S15.5} one obtains upper bounds
 on $\omega$.

 The key step in this method is to identify a highly symmetric subtensor $SP_d$ of $T_{A^{\otimes d}}$ and {to prove} that it restricts to the {above} direct sum of matrix multiplications. It turns out that $SP_d$ in many cases is not only a restriciton of the tensor associated to ${A^{\otimes d}}$, but also of its much smaller Veronese subalgebra. Thus, obtaining bounds on border rank or smoothable rank of tensors associated to Veronese subalgebras of smoothable algebras would lead to new upper bounds on $\omega$. In particular, if $SP_d$ was of minimal asymptotic rank, then $\omega=2$.

For the second point one considers some of the well-known NP-hard problems. {Very recent work \cites{bjorklund2023asymptotic, pratt2023stronger,
bjorklund2024chromatic} identifies certain} families of tensors for which
upper bounds on rank would provide algorithms of unexpectedly low
complexity. For example
for a family of tensors $SP_d$ {one could check, in the randomized setting, whether} there exist three sets from a given three (balanced) families whose union is the whole given set of size $n$, in time $C^n$ for {some} constant $C<2$.

In our work, we identify the properties of $SP_d$ to define the class of tensors called the \emph{sweet pieces}, see \autoref{def:sweet}. As argued above, {upper bounds on the ranks of sweet pieces are of central importance.} We provide new methods to upper bound not only border ranks, but also ranks of the sweet pieces $SP_d$, see \Cref{lem:samesweet,lem:sweetCWnobrank} and \autoref{cor:sweetCW=G}.
We apply these results to obtain new, best upper bounds for rank and border
rank of $SP_d$ tensors that appear in the Coppersmith-Winograd tensor
and in the work of K.~Pratt \cite{pratt2023stronger}, see below.
\begin{prop}[\autoref{prop:sweetRank}]\label{prop:introCW}
{Given nonnegative parameters $p,q$ such that $3p+3q=1$ one considers the sweet piece for the $N$-th Kronecker power of the $n\times n\times n$  Coppersmith-Winograd tensor $CW_n$. The rank of any such sweet piece} is {smaller} than
    \[n^N-\binom{N}{(2p+2q)N+1}(n-1)^{(p+q)N-1}.\] In particular, it is
    {smaller} than the border rank of $(CW_n)^{\boxtimes N}$.
\end{prop}
\begin{prop}[\autoref{prop:Prattimproved}]
Consider the tensors associated to the bilinear multiplication map
\[({\CC[x]/(x^2)}^{\otimes 3k})_k\times ({\CC[x]/(x^2)}^{\otimes 3k})_k\rightarrow ({\CC[x]/(x^2)}^{\otimes 3k})_{2k}\]
where the subindex denotes the degree (these are precisely the tensors from
\cite{pratt2023stronger}).
These tensors have rank at most
\[\frac{1}{2}8^{k}-\sum_{i=k+1}^{\lfloor 3k/2\rfloor} \binom{3k}{2i}.\]
\end{prop}

The improvements in the upper bounds {may not seem large.} Indeed, to obtain new estimates {for} $\omega$, one needs to obtain in \autoref{prop:introCW} an upper bound of {the form} $n^{CN}$, where $C<1$. {Nevertheless}, we believe that improving the upper bounds of sweet pieces is important, as e.g.~for the set covering problem \emph{subexponential} improvements may already refute the conjectures about the complexity of the problem \cite{pratt2023stronger}*{Corollary 1.12}. {In addition,} recently new barrier results appeared \cites{CVZ21, blaser2020slice}, {which show that it is essentially impossible to prove $\omega=2$ using only tensors from a class that includes} the Coppersmith-Winograd tensor. However, these barrier results are not valid, if one improves upper bounds on {the} ranks of the sweet pieces. Experts may notice that the sweet pieces of interest have asymptotically largest subrank.

\section{\normalfont \scshape\fontfamily{ptm}\selectfont Preliminaries}\label{sec:prelims}
\noindent
In the study of tensor decomposition, the quantity that has appeared most frequently in the literature is the \textit{Waring rank} of a
symmetric tensor or, equivalently, of a homogeneous polynomial. For any $f\in
S^dV$, the Waring rank, or simply the rank, of $f$ is {the} minimal number of linear
forms such that $f$ can be expressed as a linear combination of the $d$-th
powers of such forms. Over the years, many other concepts related to
symmetric tensors have been introduced, such as border rank, cactus rank, and
smoothable rank. In general, important tools for tensor decompositions are provided by the theory of apolarity.
In this article we work over an algebraically closed field $\bbK$ of
characteristic zero.
\subsection{Apolarity} \noindent Let $V$ be an arbitrary $(n+1)$-dimensional vector space over $\bbK$.
Let
\[
\calRgl\coloneq S(V),\qquad \calDgl\coloneq S(V^*)
\]
denote the symmetric algebra of $V$ and its dual, respectively.
 {Some zero-dimensional projective schemes below will live in in $\mathbb{P}(V) = \Proj S(V^*) = \Proj \calDgl$.}
Let
\[ (\calRgl)_{d}\coloneq S^d V,\qquad(\calDgl)_{d}\coloneq S^d V^*\] denote the $d$-th symmetric powers of $V$ and $V^*$, respectively.
The ring $\calDgl$ acts on $\calRgl$ so that every $\alpha\in V^*$ acts as a
partial derivative.
If we fix the dual bases $\{x_0,\dots,x_n\}$ and $\{\alpha_0,\dots,\alpha_n\}$ of $V$ and $V^*$ respectively, we can identify their symmetric algebras as
\[
\calRgl=\bbK[x_{0},\dots,x_{n}],\qquad \calDgl=\bbK[\alpha_{0},\dots,\alpha_{n}].
\]
{For such fixed bases, we also consider}
\[
\calR=\bbK[x_{1},\dots,x_{n}],\qquad \calD=\bbK[\alpha_{1},\dots,\alpha_{n}].
\]
{The ring $\calD$ can be identified with the ring of functions on the affine patch $(x_0\neq 0) \subset \mathbb{P}(V)$.}
We denote the monomials in $\calRgl$ and $\calDgl$ by
\[
\bfx^{\bfm}\coloneq x_0^{m_0}\cdots x_n^{m_n},\qquad \bfalpha^{\bfm}\coloneq\alpha_0^{m_0}\cdots \alpha_n^{m_n},
\]
for any multi-index $\bfm=(m_0,\dots,m_n)\in\bbN^{n+1}$. For such an $\bfm$, {we denote}
$\bfm!\coloneq m_0!\cdots m_n!$.
The polarization map is defined for monomials $\bfalpha^{\bfk}\in\calDgl$ and $\bfx^{\bfm}\in\calRgl$ as
\[
\bfalpha^{\bfk}\circ\bfx^{\bfm}\coloneq
\begin{cases}
\dfrac{\bfm!}{(\bfm-\bfk)!}\bfx^{\bfm-\bfk}\quad &\text{if $\bfm-\bfk\geq 0$},\\[2ex]
0\quad &\text{otherwise}.
\end{cases}
\]
\begin{defn}
\label{defn:apolarity_action}
The \textit{apolarity action} of $\calDgl$ on $\calRgl$ is defined by linearly extending
the polarization maps for each component of $\calDgl$ and $\calRgl$. It is denoted by $\circ\colon \calDgl\times\calRgl\to\calRgl$, and gives $(\sigma, f)\mapsto
\sigma\circ f$.
\end{defn}
By fixing a polynomial $f\in\calRgl$, it is possible to define a function that describes the apolarity action of the space $\calDgl$ on $f$.
\begin{defn}
For any $f\in\calRgl$, the \textit{catalecticant map} of $f$ is defined as the linear map
\begin{equation}\label{eq:catalecticant}
\begin{tikzcd}[row sep=0pt,column sep=1pc]
 \Cat_f\colon \calDgl\arrow{r} & \calRgl\hphantom{.} \\
  {\hphantom{\Cat_f\colon{}}} \sigma \arrow[mapsto]{r} & \sigma\circ f.
\end{tikzcd}
\end{equation}
The \textit{annihilator} (or \textit{apolar ideal}) of $f$ is the kernel of $\Cat_f$, that is, the set
\[
\Ann(f)\coloneq\Set{\sigma\in\calDgl|\sigma\circ f=0}.
\]
\end{defn}
The annihilator of a polynomial $f$ was called \textit{principal system} by F.~S. Macaulay and its quotient
\[
\Ap(f)\coloneq\frac{{\calDgl}}{\Ann(f)}
\]
is also called the \textit{apolar algebra} of $f$.
The apolar algebra was introduced in \cite{Mac94} by
F.~S.~Macaulay, see also \cite{IK99}*{p.~XX} and~\cite{Dol12}*{p.~75}.
The above definition immediately extends to an arbitrary number of polynomials.
\begin{defn}
Let $f_1,\dots,f_r\in\calRgl$. The \textit{annihilator} of $f_1,\dots,f_r$ is the ideal
\[
\Ann(f_1,\dots,f_r)\coloneq\bigcap_{i=1}^r\Ann(f_i).
\]
The \textit{apolar algebra} of $f_1,\dots,f_r$ is the quotient space
\[
\Ap(f_1,\dots,f_r)\coloneq\frac{{\calDgl}}{\Ann(f_1,\dots,f_r)}.
\]
\end{defn}
In the particular case of a single polynomial $f\in\calRgl$, we have the
isomorphism of $\calDgl$-modules
\begin{equation}\label{eq:apolarIso}
\calDgl\circ f\simeq\Ap(f).
\end{equation}

A polynomial {$f\in \calRgl$} is \emph{concise} if it
cannot be written, possibly after a change of coordinates, using less than $n$ variables. Equivalently, the polynomial $f$ is concise if \[\Ann(f)\cap
{\calDgl}_{\leq 1} = \{0\}.\]

{The discussion above can be repeated for the subrings $\calR$, $\calD$: the apolarity action restricts from $\calDgl$, $\calRgl$ to $\calD$, $\calR$, for a polynomial $f\in\calR$ we obtain its apolar algebra $\Ap(f) = \calD/\Ann(f)$ etc.}

\subsection{Ranks of tensors}
In this section we recall several notions of rank that will be useful below.
A crucial tool for providing upper bounds for the rank is the classical
apolarity lemma.
\begin{lem}[Apolarity lemma]
\label{lem:apolarity_lemma_schemes}
Let ${F}\in (\calRgl)_{d}$
be a homogeneous {polynomial} and $Z\subseteq\bbP^n$ be a
finite scheme. Let
\[
\nu_d\colon\bbP\bigl((\calRgl)_{1}\bigr)\to\bbP\bigl((\calRgl)_{d}\bigr)
\]
be the $d$-Veronese map. The following conditions are equivalent:
\begin{enumerate}[label=(\arabic*), left= 3pt, widest=2,nosep]
\item $[{F}]\in\langle\nu_d(Z)\rangle$;
\item $I(Z)\subseteq \Ann({F})$.
\end{enumerate}
If they hold, we say that \emph{$Z$ is apolar to ${F}$}.
\end{lem}

The lemma gives rise to three notions of rank,
and also gives an alternative but equivalent definition of Waring rank.
\begin{defn}
    The \emph{cactus rank} of $F$ is the minimum degree of a scheme $Z$ apolar to
    $F$. The \emph{smoothable rank} of $F$ is the {minimal} degree of such a
    $Z$, where $Z$ is smoothable (see~\S\ref{sec:smoothability} below).
    The \emph{Waring rank} of $F$ is the minimum degree of such a $Z$, where $Z$ is a tuple of points.
\end{defn}
The cactus rank was introduced under the name of \textit{scheme length} by A.~Iarrobino and V.~Kanev (see \cite{IK99}*{Definition 5.1}).
The term \textit{cactus rank} was first used by K.~Ranestad and F.-O.~Schreyer
in \cite{RS11}, inspired by the notion of \textit{cactus variety}, introduced
by W.~Buczy\'{n}ska and J.~Buczy\'{n}ski in \cite{BB14}. The cactus rank was analyzed in several other papers, such as  \cites{Bal18, BR13, BBG19}.
The smoothable rank was also first introduced in \cite{RS11},
motivated by several results appearing in \cites{BGI11, BB14,
BGL13}. For a homogeneous polynomial ${F}$, we denote its smoothable rank by
$\smrk {F}$ and its cactus rank by $\crk {F}$.
For any ${F}$ we have inequalities
\[
    \crk {F} \leq \smrk {F}\leq \rk {F},
\]
which can be strict in general, see, for example,~\cites{BB15, huang2020vanishing}.

\subsection{Tautological schemes}\label{ssec:tautological}
\noindent In general, the construction of apolar schemes for a given form is a subtle problem.
There are several papers in the literature that treat these objects, such as
\cites{BB14,Chi06,GRV18,RS11}.

A clear and general procedure is provided by A.~Bernardi and K.~Ranestad
in~\cite{BR13} and later in \cite{BJMR18}*{section 4}. This strategy allows to relate homogeneous polynomials to their dehomogenizations.

Before proceeding further, we warn the reader about one known caveat: Theorem~3
in~\cite{BR13} is {incorrect} as stated, see corrigendum~\cite{BR24}. A counterexample is \autoref{ex:counterexampleBR}.
The source of the problem is the {lack} of constants in the last line of the proof of~\cite{BR13}*{Lemma~2},
see~\eqref{eq:correctedBRequation} below. As far
as we know, the results of~\cite{BJMR18}*{section~4} are fine.
However,~\cite{BJMR18} uses
the contraction action instead of differentiation
(see~\cite{IK99}*{Appendix~A} for a comparison of the two actions). The whole
problem may seem notational and negligible, but it is unavoidable and led us
to the introduction of twists above.
\autoref{ex:bigSquareForEncompassing} below shows that they are
necessary.

In the current article, we do not introduce the contraction action directly.
Rather, we only introduce the necessary notation and refer the reader
to~\cite{BJMR18} and~\cite{IK99}*{Appendix~A} for details.

The main result we use is the following proposition.
\begin{prop}[\cite{BJMR18}*{Corollary 4} or corrected \cite{BR13}*{Lemma~2}]
\label{prop:BJMR18_Corollary_4}
Let $F\in\calRgl = \bbK[x_0, \ldots ,x_n]$ be an arbitrary form and let
\[
    Z_{F,x_0} := \Spec \Ap(F|_{x_0=1}) \subseteq (x_0\neq 0) \subseteq \mathbb{P}^n.
\]
Then the
scheme $Z_{F,x_0}$ is apolar
to $\twist{F}$.
\end{prop}
\begin{proof}
    We follow the notation of~\cite{BR13}*{Lemma~2}.
    Let $d$ be the degree of $F$. Take a homogeneous
        $G\in\calDgl$ such
    that its dehomogenization $g = G|_{\alpha_0=1}$ is apolar to $f=F|_{x_0=1}$.
    If the degree of $G$ is higher than $d$, then
the apolarity is clear. Otherwise,
    write \[ g = g_d + \cdots + g_0,\qquad F|_{x_0 = 1} = f_d + \cdots +
    f_0,\] where $f_i$ and $g_i$ are homogeneous of degree $i$. The apolarity reduces to \[\sum_j g_j\circ f_{j+e} = 0\] for every $e$.    We compute
    \begin{equation}\label{eq:correctedBRequation}
        G\circ \twist{F} = \sum_{e}\sum_j \frac{1}{(d-(e+j))!}(\alpha_0^{d-j} \circ
        x_0^{d-(e+j)}) g_j\circ f_{e+j} = \sum_e \frac{1}{e!} x_0^e \sum_j
        g_j\circ f_{e+j} = 0.\qedhere
    \end{equation}
\end{proof}

\subsection{Smoothability}\label{sec:smoothability}
\noindent \autoref{lem:apolarity_lemma_schemes} and \autoref{prop:BJMR18_Corollary_4}
are powerful results that can be useful in determining upper bounds. To use them to determine smoothable rank, it is necessary to discuss the notion of smoothability.
We recall that the \textit{Hilbert scheme} $\Hilb_r(\bbA^n)$ of $r$ points in the
affine space $\bbA^n$ is a scheme {parameterizing} all finite
subschemes of $\bbA^n$ of degree $r$. It was first introduced by
A.~Grothendieck in \cite{Gro61} and it is a very important object in algebraic geometry, widely used in the literature (see, e.g., \cites{Str96, MS05, Ber12}
for introductions).
Denoting by $\Hilb_r^0(\bbA^n)$ the open subset of  $\Hilb_r(\bbA^n)$ consisting of the $r$-tuples of distinct points in $\bbA^n$, we say that a scheme $Z\subseteq\bbA^n$ is \textit{smoothable} if it is contained in the closure of $\Hilb_r^0(\bbA^n)$.
The smoothability does not depend on the embedding of $Z$, only on $Z$ itself.
Moreover, a finite scheme $Z$ is smoothable if and only if all of its irreducible
components are smoothable~(see \cite{BJ17}*{Theorem~1.1}).
For more details on smoothability, the reader may consult \cites{IK99, BJ17, CEVV09,JKK19}.

\newcommand{\powerseries}{\bbK[\![t]\!]}

Let $\calD \coloneq \bbK[\alpha_{1},\dots,\alpha_{n}]$ be the coordinate ring
of $\mathbb{A}^n$.
Our interest focuses on the smoothability of symmetric powers of
algebras (see~\S\ref{sec:smoothableRank}), and we will see that it is cumbersome to consider them as
quotients of $\calD$. Thus, it is also important to consider smoothability
from the point of view of non-embedded $\bbK$-algebras.
\begin{defn}
\label{def:smoothable_algebra}
A {finite} $\bbK$-algebra $A$ is said to be \textit{smoothable} if there
exists a homomorphism $\powerseries\to \calA$, where $\powerseries$ is the
formal power series ring in $t$ and $\calA$ is a $\powerseries$-algebra
such that the \emph{special fiber} $\calA/t\calA$ is isomorphic to $A$ as an
$\bbK$-algebra,
the \emph{generic fiber} $\calA[t^{-1}]$ is a smooth
$\powerseries[t^{-1}]$-algebra, and $\calA$ is a free $\powerseries$-module.
\end{defn}
The last two definitions are equivalent thanks to the following result.
\begin{lem}[\cite{CEVV09}*{Lemma 4.1}]
\label{lem:equivalence_smoothability}
Let $I\subseteq\calD$ be an ideal such that $\dim_{\bbK} \calD/I =
r$. Then
$\Spec(\calD/I)$ is in the
smoothable component of $\Hilb_r(\bbA^n)$ if and only if $\calD/I$ is a smoothable
$\bbK$-algebra in the sense of \autoref{def:smoothable_algebra}.
\end{lem}
For any scheme $Z = \Spec(\calD/I)\subseteq\bbA^n$, much is known about smoothability of $Z$ for small values of
$\deg(Z) = \dim_{\bbK} \calD/I$. Indeed, all schemes $Z$ such that $\deg(Z)\leq 7$ are
smoothable (\cite{CEVV09}) and all Gorenstein schemes $Z$ such that $\deg(Z)\leq 13$ are
smoothable (\cite{CJN13}). However, when $\deg(Z)$ is larger the situation is
much more complicated. The known classes
of smoothable $Z$ are:
\begin{enumerate}[label=(\arabic*), left= 3pt, widest=2,nosep]
    \item $Z \subseteq \mathbb{A}^2$
    \item $Z \subseteq \mathbb{A}^3$ Gorenstein;
    \item $Z$ given by a locally complete intersection or, more generally,
        linked to a smoothable scheme;
    \item $Z$ given by a monomial ideal in an $\mathbb{A}^n$.
\end{enumerate}
Disjoint union and product of smoothable schemes are smoothable. While there
are many other constructions that give specific examples of smoothable
schemes, it seems that the above are the only general ones known. In
\S\ref{sec:smoothabilityandbrank}
 we show that being able to decide {whether} explicit, simple algebras are smoothable or not, would provide upper and lower bounds on tensor border rank that are much better than state-of-the-art techniques.

\subsection{Secant varieties and border rank}
\noindent When dealing with decompositions of homogeneous polynomials, there
is another quantity {to consider}, namely the
\textit{border rank} of a form.
The notion of border rank of an arbitrary tensor was {first introduced} in 1979 by D.~Bini, M.~Capovani, G.~Lotti, and F.~Romani in \cite{BCRL79}, and then named that way in the following year by D.~Bini, G.~Lotti, and F.~Romani in \cite{BLR80}, where the authors defined it as the minimum number
of decomposable tensors required to approximate a tensor with an arbitrarily
small error.
The analog for homogeneous polynomials is defined in the same way.
For a homogeneous polynomial ${F}$,
the \textit{border rank} of ${F}$, denoted by $\brk {F}$, is the smallest natural number $r$ such that ${F}$
is in the closure of a set of polynomials that have rank $\leq r$. The closure is taken in the Zariski topology. For $\bbK
= \bbC$ the closure {coincides} with the closure in the Euclidean topology.

Determining the border rank of a \emph{given} {homogeneous} polynomial is in general a very difficult problem.
However, again, much information is known for small values of $n$ and $d$
(see, e.g., \cites{LO13, CGLV22, GL19, LMR23, Fla23, GMR23}).
{A typical example illustrating} the difference between rank and border rank is given by a monomial of the type $x_1^{d-1}x_2$, for any $d\in\bbN$. {In fact, its rank is equal to} $d$, but its border rank is $2$, as we can express it as
\[
    x_1^{d-1}x_2=\frac{1}{d}\cdot \lim_{t\to 0}\frac{1}{t}\bigl((x_1+tx_2)^d-x_1^d\bigr).
\]

{The \textit{catalecticant rank} is a known} lower bound for the border
rank of a form ${F}\in\calRgl_d$, historically attributed to J.~J.~Sylvester
{see \cite{Syl51}}.
It states that for every $k\in\bbN$ we have
\begin{equation}\label{rel:lower_bound_border_rank}
\brk {F}\geq \rk(\Cat_{{F}})_k,
\end{equation}
where $(\Cat_{{F}})_k$ is the $k$-th catalecticant matrix,
see~\eqref{eq:catalecticant}.

The catalecticant {lower} bound, usually taken for $k\simeq d/2$, is very
{powerful} when $d$ is large with respect to $r$. For example, it gives
{set-theoretic}
equations of the $r$-th secant variety $\sigma_{r}\bigl( \nu_d(\bbP^n) \bigr)$ when $d\geq 2r$ and
$r\leq 13$, see \cite{BB14}*{Theorem~1.1 and \S8.1}.
When $r$ is large with respect to $d$, the bound is very weak
and there are few explicit classes of $f$ for which equality holds: mostly
additive decompositions, such as \[ {F}\coloneqq  x_1^d +  \ldots + x_n^d, \] and more generally forms in
disjoint sets of variables, such as tangential generalized additive
decompositions of the type \[ {F}\coloneqq  x_1^{d-1}x_{2} + x_{3}^{d-1}x_4 +  \ldots +
x_{2n-1}^{d-1}x_{2n},\]
or yet more generally the GADs, see, e.g.,~\cite{BOT23}. {All the examples known to the authors} have $\max_i\rk (\Cat_{{F}})_{i}$ equal to $n$ or slightly
higher.

There are several examples in the literature concerning the analysis of the
rank of the catalecticant matrices of certain homogeneous
polynomials (see \cite{BBM14}*{Example 2.21} for the case of monomials, see
also~\cite{IK99}*{p.~198}). A
special case concerns the powers of quadratic forms. Given any quadratic form
$q_n$ of rank equal to $n$, B.~Reznick proves that
for every integer $d\geq 1$ the catalecticant matrices of $q_n^d$ are all {of} full rank.
\begin{teo}[{\cite{Rez92}*{Theorem 8.15} using \cite{Rez92}*{Theorem~3.7 and Theorem~3.16}, see also \cite{GL19}*{Theorem 2.2}, \cite{Fla23}*{Proposition 3.2}}]
\label{teo:lower_bound_border_rank_q_n^d}
For a quadratic form $q_n\in(\calRgl)_2$ of rank $n$ and for any $d\geq 1$ we have
\[
\brk(q_n^d)\geq\binom{n+d-1}{d}.
\]
\end{teo}

A notable fact is that the border rank and the smoothable rank can be compared. Indeed, as observed by A.~Bernardi, J.~Brachat, and B.~Mourrain in \cite{BBM14}*{Remark 2.7}, using \cite{IK99}*{Lemma 5.17}, we have
\begin{equation}
\label{rel_border_rank_lower_smoothable_rank}
\brk {F}\leq \smrk {F}.
\end{equation}
There are cases where the equality in formula
\eqref{rel_border_rank_lower_smoothable_rank} does not hold. Several examples
of polynomials with border rank strictly {lower} than the smoothable rank are
given by W.~Buczy\'{n}ska and J.~Buczy\'{n}ski in \cite{BB15}.

\subsection{Lowest and top degree forms}
\noindent Some of the proofs below make use of the theory of {top and lowest degree} forms, which we briefly review here (see~\cite{Iar94}*{\S1} for details).
Write a polynomial $f\in\calR$ as \[ f = f_e + f_{e+1}+ \ldots +f_d,\]
for some $e\in\bbN$,
where
$f_i\in \calR$ is homogeneous of degree $i$ for $e\leq i\leq d$ and $f_e, f_d\neq
0$. Then the \emph{lowest degree
form} and the \emph{top degree form} of $f$ are respectively the homogeneous polynomials
\[\ldf(f) \coloneq f_e,\qquad \tdf(f) \coloneq f_d.\]
In general none of these is a monomial. Analogous definitions can be made for any polynomial
$\sigma\in \calD$.

For a $\calD$-submodule $M\subseteq \calR$, we define the
\textit{space of the lowest degree forms} and the \textit{space of the top degree forms} as the $\bbK$-linear spaces given by
\[
\ldf(M)\coloneq\langle\Set{\ldf(f)|f\in M}\rangle,\qquad \tdf(M)\coloneq\langle\Set{\tdf(f)|f\in M}\rangle,
\]
respectively.
Then
$\ldf(M)$ and $\tdf(M)$ are also $\calD$-submodules. In particular,
for any ideal $I\subseteq \calD$ such that $\calD_{\geq r}\subseteq I$ for $r{\gg 0}$, {we can define $\tdf(I)$, which is the $\bbK$-linear space spanned by the set
\[\Set{\tdf(\sigma) |\sigma\in I}.\]
It is an ideal.}

\begin{lem}
\label{lem:ann_ldf(M)=tdf(I)}
Let $M$ be a finitely generated $\calD$-submodule of $\calR$ and let $I=\Ann(M)$.
Then \[\Ann\bigl(\ldf(M)\bigr)=\tdf(I).\]
\end{lem}
\begin{proof}
{The subspace} $\ldf(M)$ is a $\calD$-submodule of $\calR$. Since $I$ annihilates $M$, also {$\tdf(I)$ annihilates} $\ldf(M)$. Therefore, we must have the inclusion \[\tdf(I)\subseteq\Ann\bigl(\ldf(M)\bigr).\]
The spaces $M$ and $\calD/\Ann(M)$ are dual, so they have the same dimension. We
obtain
\[
\dim_{\bbK}\Bigl(\calD/\Ann\bigl(\ldf(M)\bigr)\Bigr)=\dim_{\bbK}\ldf(M)=\dim_{\bbK} M=\dim_{\bbK}(\calD/I)=\dim_{\bbK}\bigl(\calD/\tdf(I)\bigr) =d,
\]
so $\tdf(I)$ and $\Ann\bigl(\ldf(M)\bigr)$ have the same (finite!) {codimension, and
so the inclusion} is an equality.
\end{proof}

\subsection{Iarrobino's symmetric decomposition}
\newcommand{\mm}{\mathfrak{m}}%
\noindent In the following, in \S\ref{sec:almostEncompassing} we will need a bit
of theory of nonhomogeneous Gorenstein algebras, as developed by
Iarrobino~\cite{Iar94}*{chapter 2}.

\newcommand{\mubar}{\overline{\mu}}%
Let $f$ be a polynomial and
$A \coloneq \Ap(f)$. We have
\[ \calD_{>\deg(f)} \subseteq \Ann(f),\] so
$A$ is {a local ring} with {the} maximal ideal $\mm$ being the image of $(\calD)_{\geq
1}$. We have a bilinear
map $\mu\colon A\times A\to \bbK$ given by
\[
    \mu(a_1, a_2) \coloneq \bigl((a_1a_2)\circ f\bigr)_0,
\]
where $(-)_0$ denotes the constant part of a polynomial and the action $\circ$ is
the apolarity action given in \autoref{defn:apolarity_action}. The whole operation is {well defined} because $\Ann(f)$
by definition annihilates $f$, so $(-)\circ f$ descends to $\Ap(f)$. For every $\bbK$-subspace $L\subseteq A$ we define
\[ L^{\perp} \coloneq \Set{ a\in A | \mu(a, L) = 0 }.\]
We have the following properties:
\begin{enumerate}[label=(\arabic*), left= 3pt, widest=2,nosep]
    \item the pairing $\mu$ is perfect;
    \item for any $m$ we have $(\mm^m)^{\perp} = (0:\mm^{m})$;
    \item for any $m$ the pairing $\mu$ induces a perfect pairing
\begin{equation}
\label{formula:mubar_perfect_pairing}
\begin{tikzcd}[row sep=0pt,column sep=1pc]
 \mubar\colon \dfrac{\mm^m}{\mm^{m+1}} \times \dfrac{(0:\mm^{m+1})}{(0:\mm^m)}\arrow[r] & \bbK\hphantom{.}\\
  {\hphantom{\mubar\colon{}}} \bigl(a_1 + \mm^{m+1}, a_2 + (0:\mm^{m})\bigr) \ar[r,mapsto] & \mu(a_1,a_2).
\end{tikzcd}
\end{equation}
\end{enumerate}
The following lemma is all we will use in the sequel. It is very {simple} in the case of homogeneous polynomials, but {subtler} in the nonhomogeneous case.

\begin{lem}\label{ref:dualElement:lem}
    Let $f\in \calR$ be a polynomial, let $m\geq 0$ be fixed and let
    $\sigma_1, \ldots ,\sigma_{l}\in \calD_{\geq m} + \Ann(f)$ be elements
    such that no nonzero $\bbK$-linear combination of $\sigma_1, \ldots
    ,\sigma_l$ lies in $\calD_{\geq m+1} + \Ann(f)$.
    Then there exist $\tau_1, \ldots ,\tau_l\in \calD$ such that {the following conditions hold:}
    \begin{enumerate}[label=(\arabic*), left= 3pt, widest=2,nosep]
        \item\label{it:firstPerp} $\deg(\tau_j\circ f)\leq m$ for every $j$;
        \item\label{it:secondPerp} $(\tau_j\sigma_j)\circ f = 1$ for every $j$ and $(\tau_i
            \sigma_j)\circ f = 0$ for every $i\neq j$.
    \end{enumerate}
\end{lem}
\begin{proof}
    \def\sigmabar{\overline{\sigma}}%
    \def\taubar{\overline{\tau}}%
    The surjection $\calD\to \Ap(f)$ sends $\calD_{\geq m}$ onto $\mm^m$.
    Consider the classes $\sigmabar_1, \ldots ,\sigmabar_l$ in $\mm^m$. By
    assumption, they yield nonzero and linearly independent elements of
    $\mm^m/\mm^{m+1}$. Complete them to a basis of $\mm^m/\mm^{m+1}$ and let
    $\taubar_1, \ldots ,\taubar_l\in (0:\mm^{m+1})$ be the duals of
    $\sigmabar_1, \ldots ,\sigmabar_l$ under the perfect pairing of formula \eqref{formula:mubar_perfect_pairing}.
    Let $\tau_1, \ldots ,\tau_l\in \calD$ be any lifts of $\taubar_1, \ldots
    ,\taubar_l$. For every $j$ we have $\mm^{m+1}\taubar_j = 0$, so
    $\calD_{\geq m+1} (\tau_j\circ f) = 0$. Hence, $\deg(\tau_j \circ f)\leq
    m$, which proves~\ref{it:firstPerp}.

    By construction of $\mubar$, the constant $\mubar(\sigmabar_i ,\taubar_j)$ is the degree
    zero part of $(\sigma_i\tau_j)\circ f$. By the choice of $\taubar_j$, we
    have
    \[\sigmabar_i\taubar_j \in \mm^m\cdot (0:\mm^{m+1}) \subseteq     (0:\mm),\] so
    $\sigmabar_i\taubar_j\mm = 0$ and hence
    $\calD_{\geq 1}\circ (\sigma_i\tau_j \circ f) = 0$. It follows that
    $\sigma_i\tau_j\circ f\in \bbK$ is a constant polynomial. We can then
    identify it with the constant $\mubar(\sigmabar_i, \taubar_j)$ and obtain
    part~\ref{it:secondPerp}.
\end{proof}

\section{\normalfont \scshape\fontfamily{ptm}\selectfont Apolar algebras of powers of
polynomials}
\noindent In this section we focus on the apolar algebras {of the powers of} of a given
polynomial $f$. It turns out that it is easier to start with the tensor
powers.

For {a number} $d\in\bbN$ take the $d$-th tensor power
$\calR^{\boxtimes d}$ and its dual space $\calD^{\boxtimes d}$. These can be
identified as the polynomial rings
\[
\calR^{\boxtimes d}\coloneq\bbK[x_{ij}\ |\ i=1,\dots,n,\,j=1,\dots,d],\qquad
\calD^{\boxtimes d}\coloneq\bbK[\alpha_{ij}\ |\ i=1,\dots,n,\,j=1,\dots,d],
\]
{respectively.}
For a polynomial $f\in \calR$, its \emph{$d$-th tensor power} $f^{\boxtimes d}$ is defined as
\[
    f^{\boxtimes d} := f(x_{11}, \ldots ,x_{n1})f(x_{12}, \ldots ,
    x_{n2})\cdots f\left( x_{1d}, \ldots ,x_{nd} \right).
\]
In the following we will use the notation $\bfx_{\bullet j}$ for $x_{1j}, \ldots,x_{nj}$ and $\alpha_{\bullet j}$ for $\alpha_{1j}, \ldots ,\alpha_{nj}$,
where $j\in \{1, \ldots ,d\}$.
\begin{prop}[apolar algebra of the tensor power]\label{prop:tensorPower}
    The apolar algebra $\Ap(f^{\boxtimes d})$ is isomorphic to
    $\Ap(f)^{\otimes d}$. In particular, it has dimension $\bigl(\dim_{\bbK} \Ap(f)\bigr)^d$.
\end{prop}
\begin{proof}
    Let $I = \Ann(f)$ and for every $j=1, \ldots ,d$, let $I_j \subseteq
    \bbK[\alpha_{\bullet j}]$ be a copy of $I$. Then every $I_j$
    annihilates $f^{\boxtimes d}$, so \[ \Ap(f^{\boxtimes d}) \coloneq \calD^{\boxtimes
    d}/\Ann(f^{\boxtimes d})\] is a quotient of
    \[
        \frac{\calD^{\boxtimes d}}{(I_1) + (I_2) + \cdots + (I_d)} \simeq
        \frac{\bbK[\alpha_{\bullet 1}]}{I_1}\otimes_{\bbK}
        \frac{\bbK[\alpha_{\bullet 2}]}{I_2}\otimes_{\bbK}
         \cdots \otimes_{\bbK}
        \frac{\bbK[\alpha_{\bullet d}]}{I_d} \simeq \Ap(f)\otimes \cdots \otimes \Ap(f) = \Ap(f)^{\otimes d}.
     \]
Moreover, if $g_1, \ldots ,g_d\in \calD\circ f$ are arbitrary partials of $f$, say
$g_i = \sigma_i\circ f$, then
\[
    g_1(\bfx_{\bullet 1})\cdots g_{d}(\bfx_{\bullet d}) =
    \bigl(\sigma_{1}(\alpha_{\bullet 1})\cdots
    \sigma_{d}(\alpha_{\bullet d})\bigr)\circ f^{\boxtimes d}.
\]
This shows that the linear space $\calD^{\boxtimes d}\circ f^{\boxtimes d}$ has
dimension at least $\bigl(\dim_{\bbK} \Ap(f)\bigr)^d$. Therefore, by formula
\eqref{eq:apolarIso}, we also have
\[
    \dim_{\bbK}\bigl(\Ap(f^{\boxtimes d})\bigr)\leq\bigl(\dim_{\bbK} \Ap(f)\bigr)^d.\qedhere
\]
\end{proof}
\autoref{prop:tensorPower} implies that the algebra
$\Ap(f^{\boxtimes d})$ is closely related to $\Ap(f)$. For example, if
$\Ap(f)$ is smoothable by a family $\bbK[\![t]\!]\to \calA$, then
$\Ap(f^{\boxtimes d})$ is smoothable by a family $\bbK[\![t]\!]\to
\calA^{\otimes d}$.

We now move on to discussing the usual powers $f^d$.
We will see that the situation here is much more
subtle. This {might be} unexpected, so we give an example. Namely, the
dimension of the $\bbK$-vector space
$\Ap(f^2)$ is by no means determined by $\dim_{\bbK}\bigl(\Ap(f)\bigr)$. In fact it is not
determined by the isomorphism class of the algebra $\Ap(f)$ itself!
\begin{exam}\label{ex:varyingDegreeOfSquare}
    Let $f \coloneq x_1^2$ and $g \coloneq x_1^2 + x_2$. We have \[\Ap(f) \simeq
    \bbK[\varepsilon]/(\varepsilon^3) \simeq \Ap(g).\] However, $\dim_{\bbK}\bigl( \Ap(f^2)\bigr)
    = 5$, while $\dim_{\bbK}\bigl( \Ap(g^2)\bigr) = 6$.
\end{exam}
We now give a bound on the dimension of $\Ap(f^d)$. It is most convenient to
work with the space $\calD \circ f^d$. Let $\calF \coloneqq \calD\circ f$.
Consider the space
{\[\calF^{d} \coloneqq \underbrace{\calF\cdots  \calF}_d \subseteq \calR,\]}
where
the product is taken $d$ times. This space contains
 \[f^d = \underbrace{f\cdots f}_d\]
 and it is closed under the action of $\calD$, hence
 \begin{equation}\label{eq:containment}
    \calD\circ f^d \subseteq \calF^{d}.
\end{equation}
{There is also} a natural surjective map $S^d\calF\to \calF^{d}$ which maps a formal product
of $d$ elements of $\calF$ to the actual product. All in all we get a
diagram
\begin{equation}\label{eq:connections}
    \begin{tikzcd}
        \calF^{d} & \ar[l, two heads] S^d \calF\\
        \calD\circ f^{d}\ar[u, hook]
    \end{tikzcd}
\end{equation}
\begin{prop}[dimension of the apolar to the symmetric
    power]\label{prop:inequalityOnDuals}
    For any arbitrary polynomial $f$, let $\ell = \dim_{\bbK} \bigl(\Ap(f)\bigr)$. Then we have
    \[
        \dim_{\bbK}\bigl( \Ap(f^{d})\bigr) \leq \binom{\ell+d-1}{d}
    \]
    and equality holds if and only if both arrows in diagram~\eqref{eq:connections}
    are bijections.
\end{prop}
\begin{proof}
    Diagram~\eqref{eq:connections} {yields} inequalities
    \[
        \dim_{\bbK}\bigl(\Ap(f^{d})\bigr) \leq \dim_{\bbK} \calF^d \leq \dim_{\bbK} S^d \calF =
        \binom{\ell+d-1}{d}.\qedhere
    \]
\end{proof}

The most interesting case for us is when the dimension of $\Ap(f^d)$ is
maximal for all $d$. {In this case}, we say that $f$ has \emph{maximal growth
of powers.} For this class of polynomials, we  show that the problem of \autoref{ex:varyingDegreeOfSquare} disappears, since the dimension of $\Ap(f^d)$ is
determined by $\dim_{\bbK}\Ap(f)$, see
\autoref{cor_apolar_algebra_p^d_lower_p_otimes_d}.
For this, we need to switch from spaces of partials to quotients of $\calD$ and { we prepare for this below.}

\newcommand{\injR}{\iota_{\calR}}
\newcommand{\projR}{\pi_{\calR}}
\newcommand{\injD}{\iota_{\calD}}
\newcommand{\projD}{\pi_{\calD}}
There is a canonical injective homomorphism $\injR\colon \calR \hookrightarrow
\calR^{\boxtimes d}$ of
algebras defined on variables by
\[
    \injR\left( x_i \right) \coloneq \frac{1}{d}\sum_{j=1}^dx_{ij}.
\]
There is a corresponding projection homomorphism
$\projR\colon\calR^{\boxtimes d}\to \calR$, which satisfies the equality
$\projR\circ\injR=\id_{\calR}$
and it is defined by $\projR(x_{ij}) \coloneq x_i$ for every $i=1,\dots,n$.
We do the same for the space $\calD^{\boxtimes d}$.

To preserve duality, we define the maps $\injD\colon \calD\to \calD^{\boxtimes
d}$ and $\projD\colon \calD^{\boxtimes d}\to \calD$ such that
\[
    \injD(\alpha_i) \coloneq \sum_{j=1}^d\alpha_{ij},\qquad
    \projD(\alpha_{ij}) \coloneq \frac{1}{d}\alpha_i,
\]
for every $i=1,\dots,n$ and $j=1,\dots,d$.
Again, we have $\projD\circ\injD=\id_{\calD}$.
The duality alluded to above follows from the equality
\[
\injD(\alpha_i)\circ
\injR(x_k)=\biggl(\Sum_{j=1}^d\alpha_{ij}\biggr)\circ\biggl(\dfrac{1}{d}\Sum_{j=1}^dx_{kj}\biggr)=\frac{1}{d}\sum_{j=1}^d(\alpha_{ij}\circ
x_{kj})=\begin{cases}
    1, & \text{if $i = k$},\\
    0, & \text{otherwise}.
\end{cases}
\]

The most important formal property of the above maps is captured by the following
statement. The reader will probably recognize that the argument is about
the representations of the symmetric group $\Sigma_d$, but we keep the proof elementary.
\begin{lem}
\label{lem_projection_apolarity_action}
Let $f_j\in\bbK[x_{\bullet j}]$ for every $j=1,\dots,d$. For any $\sigma\in
\calD$ we have
\[
\projR\bigl(\injD(\sigma)\circ(f_1\cdots f_d)\bigr)=\sigma \circ
\projR(f_1 \cdots  f_d).
\]
\end{lem}
Before we give the proof, let us {do} an illustrative example. {Take} $d = 2$
and $\sigma = \alpha_1\alpha_2$. Then \[\injD(\sigma) = (\alpha_{11} +
\alpha_{12})(\alpha_{21} + \alpha_{22}),\] hence
\[
    \injD(\sigma) \circ f_1f_2 = (\alpha_{11}\alpha_{21}\circ f_1)f_2 +
    f_1(\alpha_{12}\alpha_{22}\circ f_2) + (\alpha_{11}\circ
    f_1)(\alpha_{22}\circ f_2) + (\alpha_{21}\circ f_1) (\alpha_{12}\circ f_2).
\]
\newcommand{\barf}{\overline{f}}%
Let $\barf_i = \projR(f_i)$.
Applying $\projR$ to the right side, we obtain
\[
    \bigl(\alpha_{1}\alpha_{2}\circ \barf_1\bigr)\barf_2 +
    \barf_1\bigl(\alpha_{1}\alpha_{2}\circ \barf_2\bigr) + \bigl(\alpha_{1}\circ
    \barf_1\bigr)\bigl(\alpha_{2}\circ \barf_2\bigr) + \bigl(\alpha_{2}\circ \barf_1\bigr)
    \bigl(\alpha_{1}\circ \barf_2\bigr) = (\alpha_1\alpha_2)\circ\bigl(\barf_1\barf_2\bigr).
\]

\begin{proof}[Proof of \autoref{lem_projection_apolarity_action}]
\def\polyspace{\calP}
Considering the diagram
\[
\begin{tikzcd}
\calD\otimes\calR^{\boxtimes d}\ar[r,hook,"\injD"]\ar[dr,"\projR"']&\calD^{\boxtimes d}\otimes\calR^{\boxtimes d}\ar[r,"\mathrm{action}"]\ar[d,"\projD\otimes\projR"]&\calR^{\boxtimes d}\ar[d,"\projR"]\\
&\calD\otimes\calR\ar[r, "\mathrm{action}"] &\calR
\end{tikzcd}
\]
we have to prove that it commutes for every $\sigma\in\calD$ and every
$f_1,\dots,f_d$ with $f_j\in \bbK[x_{\bullet j}]$.
Let $\polyspace$ be the linear space spanned by all products $f_1, \ldots
,f_d$ as above. Then $\polyspace \subseteq \calR^{\boxtimes d}$ is a
$\calD^{\boxtimes d}$-submodule, in particular it is a $\calD$-submodule.
We will prove that for every $p\in \polyspace$ and $\sigma\in \calD$ we have
\begin{equation}\label{eq:compatibility}
    \projR\left(\injD(\sigma)\circ p\right)=\sigma \circ \projR(p).
\end{equation}
First, we do it for $\sigma = \alpha_{i}$ being a linear form. It is sufficient to check it {in the case where}
$p$ is a product {$f_1\cdots  f_d$} as above.

\begin{align}
\label{rel_action_variables}
\nonumber\projR\bigl(\injD(\alpha_i)\circ (f_1\cdots
f_d)\bigr)&=\projR\Bigl((\alpha_{i1}+\cdots+\alpha_{id})\circ
\bigl(f_1(x_{\bullet 1})\cdots f_d(x_{\bullet d})\bigr)\Bigr)\\[1ex]
&\nonumber=\sum_{j=1}^d\projR\Bigl(\alpha_{ij}\circ\bigl(f_1(x_{\bullet
1})\cdots f_d(x_{\bullet d})\bigr)\Bigr)\\
&\nonumber=\sum_{j=1}^d\projR\biggl(\bigl(\alpha_{ij}\circ f_j(x_{\bullet
j})\bigr)\prod_{k\neq j}f_k(x_{\bullet k})\biggr)\\
&\nonumber=\sum_{j=1}^d\bigl(\alpha_{i}\circ
f_j(x_{\bullet})\bigr)\prod_{k\neq j}f_k(x_{\bullet})\\[2ex]
&=\alpha_i\circ \bigl(f_1(x_{\bullet})\cdots
f_d(x_{\bullet})\bigr)=\alpha_i\circ\projR(f_1\cdots f_d).
\end{align}
This proves~\eqref{eq:compatibility} {in the case where $\sigma$ is a variable}. We will now
prove~\eqref{eq:compatibility} in the case where $\sigma$ is a monomial. We do this by induction
on the degree. To do the induction step, suppose that $\tau = \alpha_i\sigma$ and that~\eqref{eq:compatibility} holds for $\sigma$. Take any $p\in
\polyspace$. We have
$\injD(\sigma)\circ p\in \polyspace$, and hence
\begin{align*}
    \tau\circ \projR(p) &= \alpha_i \circ \bigl(\sigma \circ \projR(p)\bigr)
    \stackrel{\eqref{eq:compatibility}\,\mathrm{for}\,\sigma}{=} \alpha_i
    \circ \Bigl(\projR\bigl(\injD(\sigma)\circ p\bigr)\Bigr)\\
    &\stackrel{\eqref{eq:compatibility}\,\mathrm{for}\,\alpha_i}{=} \projR\Bigl(\injD(\alpha_i)\circ
    \bigl(\injD(\sigma)\circ p\bigr)\Bigr) = \projR\bigl(\injD(\tau)\circ p\bigr).
\end{align*}
This concludes the proof of~\eqref{eq:compatibility} for $\sigma$ a monomial.
For general $\sigma$, the proof follows by linearity.
\end{proof}

\autoref{lem_projection_apolarity_action} implies the following key corollary.
\begin{cor}
\label{cor_apolar_algebra_p^d_lower_p_otimes_d}
Let $f\in\calR$ and $d\geq 1$. Take $\ell = \dim_{\bbK} \bigl(\Ap(f)\bigr)$. Suppose that
\[
    \dim_{\bbK} \bigl(\Ap(f^d)\bigr) = \binom{\ell+d-1}{d}.
\]
Then $\Ap(f^d)$ is isomorphic to $S^d \Ap(f)$.
\end{cor}
\begin{proof}
Let $\sigma\in\calD$. Then
\[
\projR\bigl(\injD(\sigma)\circ f^{\boxtimes d}\bigr)=\sigma\circ f^d.
\]
In particular, we have \[\calD\cap\Ann(f^{\boxtimes d})\subseteq\Ann(f^d).\]
By the construction of $\injD$, the image of the map \[\calD\to \calD^{\boxtimes d}\to \Ap(f^{\boxtimes d})\]
lands in the $\Sigma_d$-fixed part, so we obtain the following diagram
\begin{equation}\label{eq:triangle}
    \begin{tikzcd}
        \dfrac{\calD}{\calD\cap\Ann(f^{\boxtimes d})} \ar[r, hook]\ar[d,
        two heads] & \Ap(f^{\boxtimes d})^{\Sigma_d}\\
        \Ap(f^d)
    \end{tikzcd}
\end{equation}
By \autoref{prop:tensorPower} the algebra $\Ap(f^{\boxtimes d})$ is
isomorphic to $\Ap(f)^{\otimes d}$. The $\Sigma_d$-action on $\Ap(f^{\boxtimes
d})$ corresponds to permuting factors in $\Ap(f)^{\otimes d}$, so that
\[\Ap(f^{\boxtimes d})^{\Sigma_d}  \simeq S^d\bigl(\Ap(f)\bigr)\] as algebras. We have
\[\dim_{\bbK} S^d\bigl(\Ap(f)\bigr) = \binom{\ell+d-1}{d}.\] By assumption, this is also the
dimension of $\Ap(f^d)$. Therefore, both arrows in diagram~\eqref{eq:triangle} are
isomorphisms. The claim follows.
\end{proof}

\section{\normalfont \scshape\fontfamily{ptm}\selectfont Encompassing polynomials}\label{sec:encompassing}
{The inequality} in \autoref{prop:inequalityOnDuals} and the consequence of equality
in \autoref{cor_apolar_algebra_p^d_lower_p_otimes_d}
naturally {lead to the question: for which polynomials the equality holds?} In this section we answer this question by proving
\autoref{ref:intro:growth}. The key property for us
is given by the following definition.
\begin{defn}\label{defn:encompassing}
    A polynomial $f\in\calR$ is \emph{encompassing} if there is no nonzero element $g\in\calD
\circ    f$ such that $g_{\leq 1} = 0$. Equivalently, for any basis $\{g_1, \ldots
    ,g_\ell\}$ of $\calD\circ f$, the degree $\leq 1$ parts $(g_1)_{\leq 1}$, \ldots ,
    $(g_\ell)_{\leq 1}$ are linearly independent.
\end{defn}
We observe how this notion translates to the
apolar side.
\begin{lem}\label{lem:encompassingInDualVariables}
    A polynomial $f$ is encompassing if and only if the image of
    $\calD_{\leq 1}$ in $\Ap(f)$ spans this algebra (as a $\bbK$-vector space).
\end{lem}
\begin{proof}
    Let $I \coloneq \Ann(f)$ and let $\ell := \dim_{\bbK} \bigl(\Ap(f)\bigr)$. The dimension of the image $\im(\calD_{\leq 1})
    \subseteq \Ap(f)$ is \[1+n - \dim_{\bbK} (I\cap \calD_{\leq 1}).\]
    A {polynomial} is linear if and only if its top degree form is linear, hence \[1+n-\dim_{\bbK}(I\cap \calD_{\leq 1}) = 1+n - \dim_{\bbK} \bigl(\tdf(I)\bigr)_{\leq 1}.\]
    By \autoref{lem:ann_ldf(M)=tdf(I)} we have
    \[
1+n-\dim_{\bbK} \bigl(\tdf(I)\bigr)_{\leq 1} = \dim_{\bbK} \bigl(\ldf(\calD \circ f)\bigr)_{\leq 1}.
    \]
    The polynomial $f$ is encompassing {if and only if} \[\dim_{\bbK} \bigl(\ldf(\calD \circ
    f)\bigr)_{\leq 1} = \ell.\] This {is true if and only if} the dimension of the image
    of $\calD_{\leq 1}$ is $\ell$.
\end{proof}

For {a finite dimensional} vector space $V\subseteq\calR$, we say that $V$ is
\textit{homogeneously algebraically independent} if for any basis
$\{v_1,\dots,v_k\}$ of $V$ there is no nonzero \emph{homogeneous} polynomial
$p\in\bbK[t_1,\dots,t_k]$ such that $p(v_1,\dots,v_k)=0$. {We note} that if this condition holds for some basis of $V$, then it holds for every basis of $V$.
 The only spaces we
will consider are spaces of partial derivatives and they always contain the
element $1\in \calR$. Taking $v_1 = 1$, there is always a trivial
(nonhomogeneous!) algebraic
dependence $p(t_1, \ldots ,t_k) = t_1-1$, so we {have} to restrict to
homogeneous polynomials. {This} is mostly an aesthetic choice: if $\{v_1 = 1,v_2, \ldots ,v_k\}$ is
a basis of $V$, then {the homogeneous} algebraic independence of $V$ is equivalent
to {the usual} algebraic independence of $\{v_2, \ldots ,v_k\}$.

\begin{lem}\label{lem:encompassingImpliesHomAlgIndep}
    For an encompassing polynomial, its partial derivatives are homogeneously
    algebraically independent.
\end{lem}
\begin{proof}
    Fix a basis $h_1, \ldots , h_{\ell}$ of the space of partials, where
    $h_{\ell}
    = 1$ and $h_1, \ldots , h_{\ell-1}$ have zero constant term. By assumption,
    their linear terms are linearly independent, so that, {possibly} after a
    linear change, we have $h_i \equiv x_i\bmod \calR_{\geq 2}$ for $i=1,2,
    \ldots ,\ell-1$. Assume that these partials are homogeneously
    algebraically dependent and let $\Phi$ be a homogeneous polynomial of
    degree $d$ such that
    \[
        \Phi(h_1, \ldots ,h_{\ell-1}, h_{\ell}) = 0.
    \]
    Let $\varphi \coloneq \Phi|_{h_{\ell}=1}$, then $\varphi(h_1, \ldots
    ,h_{\ell-1}) = 0$. Let $d'$ be the smallest natural number such that
    $\varphi_{d'} \neq 0$. Then
    \[
        0 = \varphi_{d'}\bigl(\ldf(h_1), \ldots ,\ldf(h_{\ell-1})\bigr) =
        \varphi_{d'}(x_1, \ldots ,x_{\ell-1}),
    \]
    so $\varphi_{d'}$ is the zero polynomial, which is a contradiction.
\end{proof}

\begin{prop}\label{prop:encompassingImpliesMaximalGrowth}
    Let $f$ be an encompassing polynomial, let $\ell \coloneq \dim_{\bbK} \Ap(f)$ and suppose that
    \[
        \ldf(\calD\circ f) = \langle l_0 = 1, l_1, \ldots , l_{\ell-1}\rangle
    \]
    for linearly independent linear forms $l_1, \ldots , l_{\ell-1}$. Then for
    every $d\geq 1$ we have
    \[
        \ldf(\calD\circ f^d) = \bbK[l_1, \ldots ,l_{\ell-1}]_{\leq d}.
    \]
    In particular, \[\dim_{\bbK} (\calD\circ f^d) = \binom{\ell+d-1}{d}.\]
\end{prop}
\begin{proof}
    The case $\deg(f) = 0$ is trivial, and in the following we assume $\deg(f)\geq 1$.
    {First, we simplify the notation.}
    Change the coordinates in $\calD = \bbK[\alpha_1, \ldots , \alpha_n]$ so that
    the kernel of $\calD_{\leq 1}\to \Ap(f)$ is generated by
    $\alpha_{\ell}, \ldots ,\alpha_n$. Then $f$ {is} in the subring $\bbK[x_1, \ldots
    ,x_{\ell-1}]$. We restrict to this subring, so that $n$ becomes $\ell-1$
    and \[\langle l_1, \ldots , l_{\ell-1} \rangle = \langle x_1, \ldots
    ,x_{n}\rangle.\]
    We make another {important} observation. Since $\calD_{\leq 1}$ surjects
    onto $\Ap(f)$, there is a linear form $\sigma\in \calD_{\leq 1}$ such that
    $\sigma\circ f = 1$. Since \[\deg(\sigma\circ f) = 0 < \deg(f),\] we have
    $\sigma\in \calD_{1}$. By changing the coordinates we can assume $\sigma = \alpha_n$.
    We have
    \[
        \alpha_n\circ(\calD_{\geq 1} \circ f) = \calD_{\geq 1} \circ (\alpha_n
        \circ f) = 0,
    \]
    so no polynomial in $\calD_{\geq 1}\circ f$ depends on $x_n$.
    We want to prove that
    \[
        \ldf( \calD\circ f^d) = \bbK[x_1, \ldots
        ,x_{n}]_{\leq d}.
    \]
    First we prove the inclusion "$\supseteq$". Assuming that this is not the case, we have by \autoref{lem:ann_ldf(M)=tdf(I)} a polynomial $\Phi\in \Ann(f^d)$ of degree at most $d$.
    Let $d' = \deg(\Phi)$ and let $\Phi_{d'}$ be its top
    degree form.
    \newcommand{\oursum}{\sum_{\bfa} \lambda_{\bfa} (\alpha_1\circ f)^{\bfa_1}
    \cdots (\alpha_{n}\circ f)^{\bfa_n}}
    For any $\bfa\in \mathbb{N}^{n}$, we write $\alpha^{\bfa} :=
    \alpha_1^{\bfa_1} \cdots \alpha_{n}^{\bfa_{n}}$.
    Take any monomial $\alpha^{\bfa}$ of degree $d'$. Then
    \begin{equation}\label{eq:action}
        \alpha^{\bfa} \circ f^d = \left( \alpha_1^{\bfa_1}  \cdots
        \alpha_{n}^{\bfa_{n}} \right)\circ f^d \equiv
        \frac{d!}{(d-d')!}(\alpha_{1}\circ f)^{\bfa_1}  \cdots (\alpha_{n}\circ
        f)^{\bfa_{n}}\cdot f^{d-d'} \bmod{f^{d-d'+1}}.
    \end{equation}
    If instead $\alpha^{\bfa}$ has degree strictly less than $d'$, then
    \[
    \alpha^{\bfa} \circ f^d \equiv 0 \bmod{f^{d-d'+1}}.
    \]
    Suppose that $\Phi_{d'} = \sum_{\bfa}
    \lambda_{\bfa}\alpha^{\bfa}$.
    {By using~\eqref{eq:action} repeatedly,} we get
    \[
        0 = \Phi\circ f^{d} \equiv \Phi_{d'}\circ f^{d} \equiv \frac{d!}{(d-d')!}\oursum f^{d-d'}\bmod{f^{d-d'+1}}.
    \]
    It follows that the sum
    \begin{equation}\label{eq:oursum}
        \oursum
    \end{equation}
    is a multiple of $f$. However, the sum does not depend on $x_n$, while {every}
    non-zero element of
    $\calR\cdot f$ depends on $x_n$. We obtain that the
    sum~\eqref{eq:oursum} is zero. By
    \autoref{lem:encompassingImpliesHomAlgIndep} this implies $\Phi_{d'}\equiv
    0$, which is a contradiction. This proves that $\Phi$ does not exist and hence we {obtain one inclusion}. The equality then follows by  \autoref{prop:inequalityOnDuals}, which completes the proof.
\end{proof}

\begin{proof}[Proof of \autoref{ref:intro:growth}]
    {We begin with a preliminary remark.}
    Since $f$ is
    concise, the map $\calD_{1}\to \Ap(f)$ is injective. The image of this map
    lies in the maximal ideal of $\Ap(f)$, {so $\calD_{\leq 1}\to
    \Ap(f)$ is also injective.} It follows that
    \begin{equation}\label{eq:manyPartials}
        n+1 = \dim_{\bbK} \calD_{\leq 1}\leq \dim_{\bbK} \Ap(f) = \ell.
    \end{equation}
   Now we begin the proof of {equivalences.}
    \autoref{prop:encompassingImpliesMaximalGrowth}
    and
    \autoref{lem:encompassingImpliesHomAlgIndep}
    prove that
    \ref{it:growthFour} implies~\ref{it:growthOne} and~\ref{it:growthThree},
    respectively.
    Let us prove that \ref{it:growthOne} implies~\ref{it:growthThree}. For
    every $d$, let $\calP_d\subseteq \calR$ denote the linear space spanned by products
    $\prod_{i=1}^d(\sigma_i\circ f)$ over all $\sigma_{i}\in \calD$.
    We have a natural surjection $S^{d} (\calD\circ f)\to \calP_d$ and an
    inclusion $\calD\circ f^d\subseteq \calP_d$, see~\eqref{eq:containment}.
    Therefore
    \[
        \binom{\ell+d-1}{d} = \dim_{\bbK} S^d (\calD\circ f) \geq \dim_{\bbK} \calP_d
        \geq \dim_{\bbK} (\calD\circ f^d).
    \]
    By assumption~\ref{it:growthOne}, {equalities} must hold. But the first
    equality shows that the map $S^d(\calD \circ f) \to \calP_d$ is an
    isomorphism, which means that the partials of $f$ {do not satisfy any} homogeneous
    polynomial of degree $d$. Iterating over all $d$, we
    obtain~\ref{it:growthThree}.

    Now we prove that \ref{it:growthThree} implies~\ref{it:growthFour}. This
    is perhaps the most surprising part, because it follows easily and yet gives a
    strong characterization. Since the partials of $f$ are homogeneously
    algebraically independent, there are at most $n+1$ of them, so $\dim_{\bbK}
    \Ap(f)\leq n+1$. By~\eqref{eq:manyPartials} we have equality $\dim_{\bbK} \Ap(f)
= n+1$ and
    additionally {we get} that
    the image of $\calD_{\leq 1}$ spans $\Ap(f)$.
    By
    \autoref{lem:encompassingInDualVariables}, this proves that $f$ is
    encompassing, hence~\ref{it:growthFour}.

    Finally, let us prove that \ref{it:growthTwo} is equivalent to~\ref{it:growthThree}.
    Assume~\ref{it:growthTwo}.
    By~\eqref{eq:manyPartials} we have $\ell \geq n+1$.
        By~\ref{it:growthTwo}, we have $\ell-1 \leq n$. Joining these, we get $\ell = n+1$. Let
    $b_1, \ldots ,b_\ell$ be a basis of the
    space of partials of $f$, where \[b_i \coloneq \frac{\partial f}{\partial
        \bfx^{\bfa}_i}\] for $i=1,2, \ldots ,\ell-1$ and $b_{\ell} = 1$.
        Suppose that $b_{1}, \ldots , b_{\ell}$ satisfy some nonzero
        homogeneous polynomial $\Phi$. Then $\Phi|_{b_{\ell} = 1}$ is
        {not constant}. This polynomial shows that
        $b_1, \ldots , b_{\ell-1}$ are algebraically dependent. But this means
        that the subfield
        \begin{equation}\label{eq:fieldExtension}
            \bbK(b_1, \ldots ,b_{\ell-1})\subseteq \bbK(x_1, \ldots
            ,x_{n})
        \end{equation}
        has transcendence degree at most $\ell-2$.
        However, by assumption~\ref{it:growthTwo}, the field
        extension~\eqref{eq:fieldExtension} is finite, so both fields have
        transcendence degree $n = \ell-1$. This contradiction
        proves~\ref{it:growthThree}.
        The proof that \ref{it:growthThree} implies \ref{it:growthTwo} is
        analogous.
\end{proof}

\subsection{Encompassing polynomial from any
polynomial}\label{sec:everyPolynomialCanBeEncompassing}

In this section we discuss how to construct an encompassing polynomial from any
polynomial and prove \autoref{ref:intro:universalityTheorem}. We keep the
notation as in the theorem, so
we {are dealing} with polynomial rings
in $k$ variables, rather than in $n$ variables.
Take a nonhomogeneous \emph{concise} polynomial $f\in \bbK[x_1, \ldots ,x_k]$ with
$n\coloneq\dim_{\bbK} \Ap(f)$ and fix
{elements}
\[
    1,\alpha_1, \ldots ,\alpha_k, \sigma_1, \ldots ,\sigma_{n-k-1}\in \calD= \bbK[\alpha_1, \ldots ,\alpha_k]
\]
restricting to a basis of $\Ap(f)$. Let $I' \coloneq \Ann(f)\subseteq \calD$ be the
annihilator of $f$.

\newcommand{\lcalR}{\hat{\calR}}%
\newcommand{\lcalD}{\hat{\calD}}%
Let us now introduce the enlarged rings
\[
    \lcalR = \bbK[x_1,\dots,x_k,y_1,\dots,y_{n-k-1}], \qquad
    \lcalD = \bbK[\alpha_1,\dots,\alpha_k,\beta_1,\dots,\beta_{n-k-1}].
\]
We define a ring automorphism $\varphi\colon\lcalD\to\lcalD$, such that
\[
\varphi(\alpha_i)=\alpha_i,\quad \varphi(\beta_j)=\beta_j-\sigma_j,
\]
for every $i=1,\dots,k$ and $j=1,\dots,n-k-1$ and consider the ideal
\[
I \coloneq I'\cdot \lcalD + (\beta_1 - \sigma_1,\dots,\beta_{n-k-1} -
\sigma_{n-k-1})=\varphi\bigl(\Ann(f)\bigr).
\]

\begin{lem}\label{lem:summary_of_coordinate_change}
    The ideal $I$ is the annihilator of an encompassing polynomial $g$ equal
    to
    \[
        g=\sum_{\bfa\in \mathbb{Z}_{\geq 0}^{n-k-1}} \frac{\bfy^{\bfa}}{\bfa!}(\sigma^{\bfa}\circ f).
    \]
    Moreover, the apolar algebras $\Ap(g)$ and $\Ap(f)$ are isomorphic.
\end{lem}
\begin{proof}
    The algebras $\lcalD/I$ and $\lcalD/\Ann(f) \simeq \Ap(f)$ are isomorphic
    via $\varphi$, so by Macaulay's inverse systems (see \cite{IK99}*{Appendix~A}), we get that $I = \Ann(g)$ for
    some non-unique $g$ and $\Ap(g) \simeq \Ap(f)$. The classes of the elements \[
    1,\alpha_1, \ldots ,\alpha_k,
    -\sigma_1, \ldots ,-\sigma_{n-k-1}\]
     in $\lcalD/\Ann(f)$ are by definition linearly
    independent. Since every $\beta_j$ annihilates $f$, these
    classes coincide with the classes of
    \[
        1,\alpha_1, \ldots ,\alpha_k, (\beta_1 - \sigma_1), \ldots
        ,(\beta_{n-k-1} - \sigma_{n-k-1}).
    \]
    But these are the images under $\varphi$ of the classes of $1,
    \alpha_1, \ldots ,\alpha_k, \beta_1, \ldots ,\beta_{n-k-1}$. Therefore,
    the classes of $\hat{\calD}_{\leq 1}$ are linearly independent, so $g$ is
    encompassing by \autoref{lem:encompassingInDualVariables}. Finally, the
    explicit expression for $g$, again non-unique, is given
    by~\cite{Jel17}*{Proposition 2.12}.
\end{proof}

\begin{exam}
    Let \[f \coloneqq x_1^4 + x_2^4+\cdots+x_n^4.\] The Hilbert function of $\Ap(f)$ is
    $(1,n,n,n,1)$ and we can take
    \[ \sigma_i \coloneqq \alpha_i,\quad \sigma_{n+i} \coloneqq \frac{1}{12}\alpha_i^2,\quad \sigma_{2n+i} \coloneqq
    \frac{1}{24}\alpha_i^3\] for every $i=1,2, \ldots ,n$, and
    \[ \sigma_{3n+1} \coloneqq \frac{1}{24}\alpha_1^4.\]
Then we have
\[ g = f(x_1+y_1, \ldots ,x_n+y_n) + \sum_{i=1}^{n} y_{n+i}(y_{n+i} +
        x_i^2) + \sum_{i=1}^n y_{2n+i} x_i + y_{3n+1}.
\]
Unlike the examples in the introduction, here $g$ has parts which
are quadratic in the $y$-variables.
\end{exam}

\begin{proof}[Proof of \autoref{ref:intro:universalityTheorem}]
    Let a concise $F\in \bbK[x_0, \ldots ,x_k]$ and let $f =
    F|_{x_0=1}$ is its dehomogenization. Let $g$ be an encompassing polynomial
    obtained from $f$ as above and let $G$ be its homogenization {multiplied} by a power of $x_0$ so that $\deg G=\deg F$ and \[G
    \in\bbK[x_0, \ldots ,x_k,y_1, \ldots ,y_{n-k-1}] \simeq \bbK[x_0, \ldots
    ,x_n].\]
    The polynomial $g$ restricts to $f$ by setting $y_{\bullet} = 0$, so {also
    $G$} restricts to $F$ by setting $y_{\bullet} = 0$.
    We now observe that the homogeneous polynomial $G$ satisfies the assumptions of
    \autoref{ref:intro:mainThmForms}.  The algebra $\Ap(g)$ is smoothable,
    since it is isomorphic to $\Ap(f)$, and the polynomial $g$ is
    encompassing, see \autoref{lem:summary_of_coordinate_change} for both
    claims. This {concludes} this part of the proof.

  Fix any $d$. From \autoref{ref:intro:mainThmForms} it follows that the
  \[
      \smrk\bigl(\twist{G^d}\bigr) \leq\binom{n+d}{d},
  \]
  so we also have by~\eqref{rel_border_rank_lower_smoothable_rank} that
    \[
        \brk\bigl(\twist{G^d}\bigr)\leq\binom{n+d}{d}.
    \]
    Since $\twist{F^d}$ is a
    restriction of $\twist{G^d}$ and border rank does not increase under restriction,
    we also get \[\brk\bigl(\twist{F^d}\bigr)\leq\binom{n+d}{d}.\]
    Finally, the explicit form~\eqref{eq:exponent} was already noted in
    \autoref{lem:summary_of_coordinate_change}.
\end{proof}

\subsection{Almost encompassing polynomials}\label{sec:almostEncompassing}

\begin{defn}
    A polynomial $f$ is \emph{almost encompassing} if $f_{\leq 1} = 0$ but for
    every nonzero element $g$ of $\calD_{\geq 1}\circ f$ we have $g_{\leq 1} \neq 0$.
\end{defn}
For example, {every} quadratic form is an almost encompassing polynomial. In
general, for any encompassing polynomial $f$, the polynomial $f_{\geq 2}$ is
almost encompassing. This shows that {many} examples of almost encompassing
polynomials can be produced using the method
{of}~\S\ref{sec:everyPolynomialCanBeEncompassing}.

In this section we show that for an almost encompassing polynomial $f$, {its $\deg(f)$-th power does not have the maximal dimension of the apolar algebra}, i.e.,
\[
\dim_{\bbK} \Ap(f^{\deg f})<\binom{\deg f-1+\dim_{\bbK}\Ap(f)}{\deg f}.
\]
We conclude with a similar statement for arbitrary polynomials.
\begin{teo}\label{ref:almostCompassing:teo}
    Let $f$ be a non-constant polynomial with zero linear part: $f_1 = 0$. Then $f$ is equal to a
    polynomial of degree $\deg f$ in the variables $\calD_{\geq 1}\circ f$.
\end{teo}
\begin{proof}
    \def\lin{\mathcal{L}}%
    Take $f\in\calR$. After a possible coordinate change and
    lowering the $n$, we can assume that there is no linear form in the variables $\alpha_1, \ldots
    ,\alpha_n$ annihilating $f$.
    For every $j\geq 2$ define a linear subspace of $\calD_{1}$ by the formula
    \[
        \lin_{\geq j} = \Set{ \beta\in \calD_{1} |  \exists\, \sigma\in
        \calD_{\geq j} : \beta\circ f = \sigma \circ
    f} = \calD_{1}\cap \bigl(\Ann(f) + \calD_{\geq j}\bigr).
    \]
    We have inclusions
    \[ \lin_{\geq 2}\supseteq \lin_{\geq 3} \supseteq
    {\cdots} \supseteq \lin_{\geq j} \supseteq
    {\cdots}, \] so after another coordinate change we may assume
    that for every $j$ we have an $n_j$ such that
    \begin{equation}\label{eq:basisChoice}
        \lin_{\geq j} = \langle \alpha_{n_j+1}, \ldots
        ,\alpha_n\rangle
    \end{equation}
    Taking $k \coloneqq n_2$ equality~\eqref{eq:basisChoice} for $j = 2$ shows that
    \[
        \Ann(f) \subseteq \langle \alpha_{k+1}, \ldots
        ,\alpha_n\rangle + \calD_{\geq 2}.\]
        Dualizing this containment, we
        obtain
        \[
            \langle 1, x_1, \ldots ,x_k\rangle = \left(\langle \alpha_{k+1}, \ldots
            ,\alpha_n\rangle + \calD_{\geq 2}\right)^{\perp} \subseteq \bigl(
            \Ann(f) \bigr)^{\perp} = \calD\circ f,
        \]
        so $x_1, \ldots ,x_k$ are partials of $f$. Additionally, it follows
        from~\eqref{eq:basisChoice} and Nakayama's
        Lemma that the classes of $\alpha_1, \ldots ,\alpha_k$ generate the
        algebra $\Ap(f)$.

        We now introduce an auxiliary grading on $\calR$.
        Namely, {letting} $n_1 = 0$, for every $i = 1 \ldots n$ there is a unique
        $j$ such that $n_j < i \leq n_{j+1}$. We take the degree of $x_i$ to be $j$. For example, the degrees of $x_1, \ldots ,x_k$ are $1$.
        {We call the resulting grading on $\calR$ the \emph{custom grading} and the degree of an element the \emph{custom degree}}, which we denote by $\deg'$.

        For every $i\geq k+1$ we will now inductively generate elements $\sigma_i,
        \tau_i\in \calD$. Suppose we have already {generated} them for all $i \leq
        n_j$. To get them for $i=n_j+1, \ldots ,n_{j+1}$, we
        use~\eqref{eq:basisChoice} and fix $\sigma_i\in \calD_{\geq j}$ so
        that $\alpha_i - \sigma_i$ annihilates $f$.

        By definition of $\lin_{\geq j}$, $\lin_{\geq j+1}$, and the construction
        of~\eqref{eq:basisChoice}, no nonzero linear combination of the elements
        $\alpha_{n_j+1}, \ldots , \alpha_{n_{j+1}}$
        lies in $\Ann(f) + \calD_{\geq j+1}$.

        Since $\sigma_i \equiv \alpha_i \bmod \Ann(f)$ for $i=n_j+1,  \ldots ,
        n_{j+1}$, no nonzero linear combination of $\sigma_{n_j+1}, \ldots ,
        \sigma_{n_{j+1}}$ lies in $\Ann(f) + \calD_{\geq j+1}$.
        Using
        \autoref{ref:dualElement:lem}, we {get} $\tau_{n_j+1}, \ldots
        ,\tau_{n_{j+1}}$ {such that}
        \begin{enumerate}[label=(\arabic*), left= 3pt, widest=2,nosep]
            \item $\deg(\tau_i \circ f) \leq j$ for every
                $i=n_{j}+1, \ldots ,n_{j+1}$,
            \item $\sigma_i \circ (\tau_i \circ f) = 1$ for every
                $i=n_{j}+1, \ldots ,n_{j+1}$,
            \item $\sigma_k\circ (\tau_i \circ f) = 0$ for $n_{j}+1\leq
                i,k\leq n_{j+1}$ with $i\neq k$.
        \end{enumerate}
        Since $\sigma_i\circ f = \alpha_i \circ f$,
        the last two equations imply that
        \begin{equation}\label{eq:linearPart}
            \alpha_i\circ (\tau_i\circ f) = 1, \qquad
            \alpha_k\circ (\tau_i\circ f) = 0,
        \end{equation}
        for $n_{j}+1\leq i,k\leq n_{j+1}$ with $i\neq k$.

        {We claim that} for every $\rho\in \calD$ the custom degree $\deg'(\rho\circ
        f)$
        is equal to
        $\deg(\rho \circ f)$. It is {immediate} that \[\deg'(\rho\circ f) \geq
        \deg(\rho \circ f),\] as this holds for every monomial. We will now
        prove that \[\deg'(\rho\circ f) \leq
        \deg(\rho\circ f).\]
        It suffices to prove that
        $\rho\circ f$ is annihilated
        by every monomial
        $\alpha_1^{a_1} \ldots \alpha_n^{a_n}$ such that \[\sum a_i \deg'(x_i) >
        \deg(\rho \circ f).\]
        For every such monomial, we get
        \[
            (\alpha_1^{a_1} {\cdots} \alpha_n^{a_n})\circ\rho \circ f =
            (\alpha_1^{a_1} {\cdots} \alpha_k^{a_k} \sigma_{k+1}^{a_{k+1}} {\cdots}
            \sigma_{n}^{a_{n}}) \circ \rho \circ f.
        \]
        By definition of the custom degree, the element \[
        \alpha_1^{a_1} {\cdots} \alpha_k^{a_k} \sigma_{k+1}^{a_{k+1}} {\cdots}
        \sigma_{n}^{a_{n}}\] lies in $\calD_{\geq \sum a_i \deg'(x_i)}$ which is
        contained in $\calD_{>\deg(\rho\circ f)}$ so this element {does indeed annihilate} $\rho \circ f$. This {completes} the proof that $\deg'(\rho
        \circ f) =
        \deg(\rho\circ f)$.

        Fix an $i\in \{1, \ldots ,n\}$ and let $j$ be the custom degree of $x_i$, so
        that $n_j < i\leq n_{j+1}$.
        Consider the element $\tau_i\circ f$. By {the first bullet} above, the degree of $\tau_i\circ f$ is
        at most $j$. By the above argument, {the custom degree of
        $\tau_i\circ f$ is also at most $j$}. In particular, the polynomial $\tau_i \circ f$
        does not contain $x_{k}$ for $k > n_{j+1}$ and contains
        $x_{n_{j}+1}, \ldots ,x_{n_{j+1}}$ at most in the linear part.
        By~\eqref{eq:linearPart} the linear part contains only $x_{i}$, so in
        particular $\tau_i \circ f$ does not contain $x_k$ with $k>n_j$,
        $k\neq i$.

        We now prove that the polynomial $f$, written as a polynomial {in the variables}
        \[y_1\coloneqq x_1,\quad \ldots,\quad y_k:=x_k,\qquad y_{k+1}\coloneqq\tau_{k+1}\circ f,\quad \ldots ,\quad y_n\coloneqq\tau_n\circ f,\]
        is of degree equal to $\deg f$. Observe that since $f_1 = 0$, all
        these new variables lie in $\calD_{\geq 1} \circ f$.
     We
     subsequently replace $x_n, \ldots ,x_{k+1}$ by the corresponding
        {\[ y_n=\tau_n\circ f,\quad \ldots ,\quad y_{k+1}=\tau_{k+1}\circ f,\]}
        as follows. We know
        that for $i$ with $\deg'(x_i) = j$, the element $\tau_i\circ f$ is of
        the form
        \[
            x_i + P_i(x_1, \ldots ,x_{n_{j}})
        \]
        for a polynomial $P_i$.
        Going down from $n$ to $k+1$, we replace \[x_i \coloneqq \tau_i \circ f -P_i(x_1, \ldots ,x_{n_{j}})=y_i-P_i(x_1, \ldots ,x_{n_{j}}).\] We have $\deg'(x_i) = j$ and
        $\deg'(\tau_i \circ f) \leq j$, so \[\deg'\bigl(P_i(x_1, \ldots
        ,x_{n_{j}})\bigr)\leq j.\]
Therefore, after each substitution the custom
        degree remains at most $\deg f$, so the usual degree, where the usual degree of each $y_i$ {is $1$}, also remains at most
        $\deg f$.
        Finally, all variables in the top degree part of $f$ are in the subset
        of $x_1, \ldots ,x_k$, and these are unchanged, so the obtained
        polynomial is of degree at least $\deg f$. Hence, it has degree
        exactly $d$.
\end{proof}

\begin{cor}\label{ref:growthuptodegImpliesEncompassing:cor}
    Let {$f$ be a polynomial} {in $\calR$} and let $\ell = \dim_{\bbK} \Ap(f)$. Suppose that
    \[
        \dim_{\bbK} \Ap(f^d) = \binom{\ell+d-1}{d}
    \]
    for $d=1,2, \ldots ,\deg f$. Then $f$ is encompassing.
\end{cor}
\begin{proof}
    Suppose that $f$ is not encompassing. Then there is a partial $g\in
    \calD\circ f$ which
    is almost encompassing. By \autoref{ref:almostCompassing:teo} it follows
    that there is a trivial homogeneous algebraic dependence of degree
    $d := \deg(g) \leq \deg(f)$ {between} the partials of $f$. As in the proof of
    \autoref{prop:inequalityOnDuals} it follows that \[\dim_{\bbK} \Ap(f^d) <
    \binom{\ell+d-1}{d},\] which is a contradiction.
\end{proof}
\begin{rem}
    We conjecture that the bound in
    \autoref{ref:growthuptodegImpliesEncompassing:cor} is optimal in that for
    every almost encompassing polynomial $f$, we have \[\dim_{\bbK}\Ap(f^d)
    =\binom{\ell+d-1}{d}\] for all $d < \deg f$.
\end{rem}

\section{\normalfont \scshape\fontfamily{ptm}\selectfont Smoothability}
\subsection{Smoothable rank of powers and locally encompassing polynomials} \label{sec:smoothableRank}
\noindent In this section we relate the previous results to the ``global'' setting of
border ranks and smoothable ranks of homogeneous polynomials. A
homogeneous polynomial $F\in \bbK[x_0, \ldots ,x_n]$ is \emph{locally encompassing}
if there exists a linear form, {say $x_0$,} such that $F|_{x_0 = 1}$ is
encompassing.

A locally encompassing form comes with the tautological apolar scheme {as defined} in~\S\ref{ssec:tautological}. To complete the proof of
\autoref{ref:intro:mainThmForms}, we need to know that such a scheme is
smoothable.

\begin{exam}
    In the setting of \autoref{ex:quadrics}, we want to know that
    \[
        \Ap\bigl( (x_1^2 +  \ldots + x_{n-1}^2 + x_n)^d \bigr)
    \]
    is smoothable for every $d\geq 1$. For $d = 1$ this algebra has Hilbert
    function $\HF = (1, n-2, 1)$ and {it is smoothable, for example,
    by~\cite{CEVV09}*{Proposition 4.9}.}

    In fact, an explicit smoothing is given over $\mathbb{A}^1$ with parameter $t$ by the
    family of $n$ points
    \[
        (0,0, \ldots ,0),
        \quad \left(t^2, t^2, \ldots, t^2, \frac{1}{2} t^3\right),
        \quad \Set{(\underbrace{0,  \ldots , 0, t}_k, 0, \dots ,0) | k=1,
        \ldots ,n-1}.
    \]
    In contrast, already for $d = 2$ the
    Hilbert function is \[\Biggl(1, n-2, \binom{n+1}{2} - 2n, n-2, 1\Biggr)\] and there is no general
    statement similar to \cite{CEVV09}*{Proposition 4.9} in the
    literature.
\end{exam}

For a {finite dimensional} $\bbK$-vector space $V$ we have the vector spaces $V^{\otimes d}$ and $S^d V
\subseteq V^{\otimes d}$ where $S^d V = (V^{\otimes d})^{\Sigma_d}$ for the
usual symmetric group action. {For a finite dimensional $\bbK$-algebra $A$, both
$A^{\otimes d}$ and $S^d A$ are also algebras:} the tensor product of
$\bbK$-algebras is a $\bbK$-algebra and the symmetric group $\Sigma_d$ acts on
$A^{\otimes d}$, whence the invariants $S^dA = (A^{\otimes d})^{\Sigma_d}$ {form}
its subalgebra.
\begin{prop}\label{ref:smoothabilityAndPowers:prop}
Let $A$ be a smoothable $\bbK$-algebra. Then $S^dA$ is smoothable for every $d\geq 1$.
\end{prop}
\begin{proof}
    Let $\powerseries\to \calA$ be a smoothing of $A$ as in
    \autoref{def:smoothable_algebra}. As a $\powerseries$-module, the algebra
    $\calA$ is isomorphic to $\powerseries^{\oplus \ell}$, where $\ell = \dim_{\bbK}
    A$.
    Consider the tensor product algebra $\calA^{\otimes d}$. As a
    $\powerseries$-module, it is free of rank $\ell^d$.
    The subalgebra $S^d \calA \subseteq \calA^{\otimes d}$ is stable under the
    $\powerseries$ action. Power series are a principal ideal domain, so a
    submodule of a free $\powerseries$-module {is also free}. Hence
    $S^d\calA$ is a free $\powerseries$-module.

The special fiber $(S^d\calA)/tS^d\calA$
is isomorphic to $S^dA$ since $S^d(-)$ is exact. The generic fiber of
$S^d\calA$
is isomorphic to $S^d(\calA[t^{-1}])$.
Recall that over any field $L$, a {finite dimensional} $L$-algebra is smooth
(over $L$) if and only if it is \'etale if and only if it is geometrically
reduced, that is, it is reduced and remains reduced after any change of basis by a
field extension $L\subseteq L'$,
see~\cite{stacks-project}*{\href{https://stacks.math.columbia.edu/tag/03PC}{Tag 03PC(1)},
\href{https://stacks.math.columbia.edu/tag/030W}{Tag 030W}, \href{https://stacks.math.columbia.edu/tag/05DS}{Tag 05DS}}.
Take $L = \powerseries[t^{-1}]$.
The $L$-algebra \[\calA^{\otimes d}[t^{-1}] \simeq (\calA[t^{-1}])^{\otimes d}\]
is smooth by \autoref{def:smoothable_algebra}. Therefore, it is geometrically
reduced, so {automatically its $L$-subalgebra $S^d \calA$ is also geometrically
reduced, and hence it is also smooth.} Thus $S^d\calA$ gives a smoothing of $S^dA$.
\end{proof}
\begin{proof}[Proof of \autoref{ref:intro:smoothability}]
    The first part is proved in \autoref{ref:smoothabilityAndPowers:prop}. {The second part follows from the first one and from}
    \autoref{cor_apolar_algebra_p^d_lower_p_otimes_d}.
\end{proof}

\begin{cor}\label{cor:smoothability_symmetric_powers_algebra}
    Let $f$ be an encompassing polynomial and take $\ell = \dim_{\bbK} \Ap(f)$.
    Suppose that $\Ap(f)$ is smoothable. Then
    for every $d$ the algebra $\Ap(f^d)$ is smoothable of
    degree $\binom{\ell+d-1}{d}$.
\end{cor}
\begin{proof}
    Fix any $d$.  Since $f$ is encompassing, by
    \autoref{prop:encompassingImpliesMaximalGrowth}, the algebra $\Ap(f^d)$ has degree
    $\binom{\ell+d-1}{d}$, so the claim follows from
    \autoref{ref:intro:smoothability}.
\end{proof}

\begin{proof}[Proof of \autoref{ref:intro:mainThmForms}]
    Let $f$ be the encompassing dehomogenization of $F$ as in the theorem.
    Since $F$ was concise, also $f$ is concise, so the natural map
    $\calD_{\leq 1}\to \Ap(f)$ is injective. It is also surjective by
    \autoref{lem:encompassingInDualVariables} hence it is bijective and
    \[
    \dim_{\bbK}
    \Ap(f) = \dim_{\bbK} \calD_{\leq 1} = n+1.
    \]

    By
    \autoref{cor:smoothability_symmetric_powers_algebra} the algebra
    $\Ap(f^d)$ is smoothable and of {the required} degree $\binom{n+d}{d}$. By
    \autoref{prop:BJMR18_Corollary_4} the smoothable rank of $F$ is at most
    $\binom{n+d}{d}$. Moreover, by
    \autoref{prop:encompassingImpliesMaximalGrowth} there is no polynomial of
    degree $d$ {that annihilates} $F^d$, which means that the rank of
    $\Cat_{F, d}$ is $\binom{n+d}{d}$. This shows that the smoothable rank
    {is equal to the catalecticant rank}. The border rank is always between the
    {catalecticant rank and the smoothable rank}, {so it is also equal to both,} and similarly for {the cactus rank and the border cactus rank.}
\end{proof}

    \begin{exam}\label[examp]{ex:bigSquareForEncompassing}
        In this example we show that
            \autoref{ref:intro:mainThmForms} would be false without the
        twists.
        Take the cubic form
        \[
            F \coloneqq x_1^3 + x_2^3 + x_0(x_1y_1 + x_2y_2) + x_0^2 y_0.
        \]
        The algebra $\Ap(F|_{x_0=1})$ is smoothable and has degree
        $6$, so $F$ satisfies the assumptions of
        \autoref{ref:intro:mainThmForms}. Moreover, both the smoothable and
        {the border} rank of $F$ are equal to $6$, so that $F$ is \emph{not} wild in
        the sense of~\cite{BB15}.
        The rank of the middle
        catalecticant for $F^2$ is $25$, so
        $\brk(F^2)\geq 25 > \binom{7}{2}$. Similarly, $\brk(F^{[2]})\geq 25$, where
        $F^{[2]}$ is the divided square, see~\cite{IK99}*{Appendix~A}.
    \end{exam}

\begin{proof}[Proof of \autoref{ref:quotients:prop}]
    The proof follows the one of \autoref{ref:smoothabilityAndPowers:prop}, so
    we only sketch it. Let $\calZ = \Spec(\calA)$ for $\calA$ as in
    \autoref{def:smoothable_algebra}. The quotient $\calZ \goodquotient G$
    corresponds to the subalgebra $\calA^{G}$ of $\calA$.
    We argue as in \autoref{ref:smoothabilityAndPowers:prop}, replacing
    $\Sigma_d$, $\calA^{\otimes d}$, $S^d \calA$ by $G$, $\calA$, $\calA^G$,
    respectively.
\end{proof}

\begin{exam}\label{ex:squareNotSmoothable}
    Let \[f \coloneqq x_3^3+x_1x_2x_4+x_3x_4^2+x_2^2x_5+x_2x_3x_5+x_1x_5^2+x_5^3\] be a
    cubic form in $5$ variables where, as always, $\bbK$ has characteristic
    zero. The algebra $\Ap(f)$ is
    smoothable by~\cites{BCR22, Jel14}.
    It has degree $12$. The algebra $\Ap(f^2)$ has degree $67 <
    \binom{12+1}{2}$. Moreover it satisfies the \emph{trivial negative
    tangents} condition from~\cite{Jel19}*{Theorem~1.2} and
    thus in particular it is not smoothable (and not cleavable).
\end{exam}
\subsection{Smoothability and border rank}\label{sec:smoothabilityandbrank}
{It is known that the smoothability} of an algebra $A$ corresponds to {the minimal} border rank of the tensor $T_A$
(see~\cites{BL16, landsberg2017abelian}). As we show below, deciding {the smoothability} of very simple algebras allows to
determine {the border} rank of arbitrary tensors up to a multiplicative factor. Our results provide another indication that deciding {whether} an algebra is smoothable should be {very hard in general.}

Given a concise tensor $T\in S^2(\bbK^n)\otimes \bbK^m$ and a {non-negative}
integer $k$ we define the standard graded local algebra $A_{T,k}$ as follows.
The algebra $A_{T,k}$ is nonzero only in {degree} $0$, $1$ and $2$.
Its {degree zero} part is spanned by $1$. Its degree one part has a basis
$x_1,\dots,x_n,y_1,\dots,y_k$. Its degree two part is $\bbK^m$.
{So} the algebra $A_{T,k}$ has degree $1+n+k+m$.
{Multiplication} is defined as follows.
{Every} $y_i$ is annihilated by all variables. The multiplication {between} $x_i$ is
defined by the linear map $S^2((\bbK^n)^*)\rightarrow \bbK^m$ induced by the tensor $T$.
\begin{exam}\label{exm:linsp}
Let $m=n=2$. {Let $T$ be represented} by a $2$-dimensional space of symmetric $2\times 2$ matrices:
\[
\begin{pmatrix}
2a+b& 3b\\
3b& 0\\
\end{pmatrix}.
\]
The multiplication tensor of the algebra $A_{T,1}$, represented as a linear space of symmetric matrices, is:
\[
\begin{pNiceMatrix}[columns-width = 18pt]
j & x_1 & x_2 &y_1 & a & b\\
x_1& 2a+b & 3b & 0  & 0 & 0\\
x_2& 3b & 0    & 0 &  0 & 0\\
y_1& 0 & 0    & 0 & 0  & 0\\
a& 0 & 0    & 0 & 0  & 0\\
b& 0 & 0    & 0 & 0  & 0
\end{pNiceMatrix}.
\]
\end{exam}
\begin{defn}
In analogy to the usual rank, for tensors $T\in S^2(\bbK^n)\otimes \bbK^m$ we define \emph{partially symmetric rank} of $T$ as the minimum $r$ so that $T$ is a sum of $r$ rank one tensors in $S^2(\bbK^n)\otimes \bbK^m$. For $\bbK=\CC$, the \emph{partially symmetric border rank} of $T$ is the minimum $r$ so that in {every} neighbourhood of $T$ there {are} tensors of partially symmetric rank $r$.
\end{defn}
\begin{prop}\label{prop:oneestimate}
For any $k\in \mathbb{N}$,
any tensor $T\in S^2(\bbK^n)\otimes \bbK^m$ is a restriction of the
tensor $T_{A_{T,k}}$ of the algebra $A_{T,k}$ associated to $T$. In particular, the border rank of $T$ is at most the border rank of  $T_{A_{T,k}}$. If $A_{T,k}$ is smoothable,
then the border rank and the partially
 symmetric border rank of $T$ {are} at most $1+n+k+m$.
\end{prop}
At first glance, {setting} $k=0$ seems optimal. However, it may well
    happen that $A_{T, 0}$ is not smoothable, while $A_{T, k}$ is smoothable
for some $k > 0$.
\begin{proof}
The tensor $T$ is a {restriction since} it corresponds to the multiplication map of
two elements of degree one in $A_{T, k}$.
If $A$ is a tuple of $e$ points, then $T_{A}$ has partially symmetric
rank $e$, so if
$A_{T,k}$ is smoothable then
$T_{A_{T,k}}$ has partially symmetric border rank $\dim_{\bbK} A_{T,k} = 1+n+k+m$.
The last
statement follows.
\end{proof}
\begin{teo}\label{thm:twoestimate}
If $T$ has partially symmetric border rank $b$, then $A_{T,b-m}$ is
smoothable.
\end{teo}
\begin{rem}
Note that for any $T \in S^2(\bbK^n)\otimes \bbK^m$ we may consider an
isomorphic nonconcise tensor $T'\in S^2(\bbK^n)\otimes \bbK^{m+1}$. For such a
tensor we have $A_{T,k+1}=A_{T',k}$. Hence, formally
$A_{T,b-m}=A_{T',b-(m+1)}$, whenever $b-(m+1)\geq 0$. {Thus, we can} naturally
extend the definition of $A_{T,b-m}$, even when $b-m<0$, which may happen only
for {non-concise} tensors. This reduces the theorem to the case {where} $T$ is
concise.
\end{rem}
\begin{proof}
    \def\Balg#1{B_{T}(#1)}

    A family of $T$ induces a flat family of $A_{T, b-m}$, so by semicontinuity
    it {suffices} to prove the result for $T$ of partially symmetric
    \emph{rank} equal to $b$. Such a tensor $T$ is given by $\sum_{i=1}^b
    L_i^{\otimes 2}\otimes z_i$, where $L_i\in \bbK^n$ and $z_i\in
    \bbK^m$ are arbitrary. {Since} $T$ has partially symmetric rank $b$, the $L_i$ are
    pairwise different.
    Consider a generic tensor of such rank: let
    $T'\in S^2(\bbK^n) \otimes \bbK^b$ be given by
    \[
        T' = \sum_{i=1}^b L_i^2\otimes \widetilde{z_i},
    \]
    where $\widetilde{z_1}, \ldots ,\widetilde{z_b}$ is a basis of $\bbK^b$.
    Then $T$ is a restriction of $T'$ by a linear map $\pi$ {which} sends every
    $\widetilde{z_i}$ to $z_i$. This map {has an} $m$-dimensional image by
    {the assumption} on the conciseness of $T$. Let $K = \ker \pi\subseteq
    \bbK^b$ and let $k_1, \ldots ,k_{b-m}$ be the basis of $K$.

    Consider the algebra $A_{T',b-m}$. {By definition, it contains the subspace $\bbK^b$ as a degree two part.}
    For every $\lambda\in \bbK$ consider the quotient
    of $A_{T', b-m}$ by the ideal
    \[
        \left( k_i - \lambda y_i\ |\ i=1, \ldots ,b-m \right).
    \]
    Denote it by $\Balg{\lambda}$.
    Both $k_i$ and $y_i$ are annihilated by the maximal ideal of $A_{T',
    b-m}$, so $\dim_{\bbK} \Balg{\lambda} = \dim_{\bbK} A_{T', b-m} - (b-m)$. Thus, $\Balg{\lambda}$ form a
    flat family. For {any non-zero $\lambda$,} taking
    the quotient is the same as evaluating $y_i \coloneqq \lambda^{-1}k_i$, so
    $\Balg{\lambda}$ is isomorphic to the algebra $A_{T', 0}$.
    We claim that $\Balg{\lambda= 0}$ is
    isomorphic to $A_{T, b-m}$. Indeed, the
    relevant part of the multiplication in $\Balg{\lambda=0}$ sends $S^2\langle x_1, \ldots
    ,x_n\rangle$ to
    {\[
        \frac{\langle \widetilde{z}_1, \ldots ,\widetilde{z}_b \rangle}{\langle k_1, \ldots ,
            k_{b-m}\rangle} \simeq \frac{\bbK^b}{K} \simeq \langle z_1, \ldots
            ,z_b\rangle.
    \]
    So $\Balg{\lambda=0}$ is indeed isomorphic to $A_{T, b-m}$.}

    We will now prove that $A_{T', 0}$ is smoothable, so $\Balg{\lambda}$ {is}
    smoothable for $\lambda\neq 0$ and, by semicontinuity
    (see~\autoref{ssec:smoothability}) {$\Balg{\lambda=0} \simeq
    A_{T, b-m}$ is also smoothable.}

    Recall that $A_{T', 0}$ is generated by $x_1, \ldots ,x_n$ and its degree
    two part is identified with $\bbK^b$, where the multiplication is induced
    by $L_1^{\otimes 2}$, \ldots , $L_b^{\otimes 2}$. Every form $L_i\in \langle x_1, \ldots
    ,x_n\rangle^*$ yields a point of the affine space $\mathbb{A}^n =
    \Spec\bbK[x_1, \ldots ,x_n]$. Consider, for every $\mu\in \bbK$ the
    subscheme of $\mathbb{A}^n$ given by
    \[
        \{\mu L_1\} \cup \left\{ \mu L_2 \right\}\cup  \ldots \cup \left\{
        \mu L_b \right\}\cup \Spec\left( \frac{\bbK[x_1, \ldots ,x_n]}{(x_1, \ldots ,x_n)^2}
        \right).
    \]
    Such a subscheme is smoothable for every $\mu$ and
    the limit of at $\mu\to 0$ {gives} $\Spec A_{T', 0}$.
\end{proof}
\begin{cor}\label{cor:estimateparsym}
For a tensor $T\in S^2(\bbK^n)\otimes \bbK^m$ let $k$ be the smallest integer
{such} that $A_{T,k}$ is smoothable. Then the partially symmetric border rank of
$T$ is at most $1+n+k$ and at least $k$.
\end{cor}
\begin{proof}
The first estimate follows from \autoref{prop:oneestimate} and the second from \autoref{thm:twoestimate}.
\end{proof}
\begin{defn}
For any concise tensor $T\in \bbK^n\otimes \bbK^n\otimes \bbK^m$ let $T_S\in S^2(\bbK^{2n})\otimes \bbK^m$ be constructed by
replacing each slice $T(e_j^*)$ by a symmetric $2n\times 2n$ matrix, {placing}
$T(e_j^*)$ on the antidiagonal blocks, {where one block is transposed in order
to obtain symmetry.} Clearly, $T_S$ restricts to $T$.
\end{defn}

\begin{cor}\label{cor:smoothabilitybrank1}
For a concise $T\in \bbK^n\otimes \bbK^n\otimes \bbK^m$ let $k$ be the smallest integer {such} that $A_{T_S,k}$ is smoothable. Then the border rank of $T$ is at most $1+2n+k$ and at least $\frac{1}{2}k$.
\end{cor}
\begin{proof}
As $T_S$ restricts to $T$, the
partially symmetric border rank of $T_S$ is at least the border rank of $T$.
We claim that the partially symmetric border rank of $T_S$ is at most twice the border rank of $T$.{Considering the limits, it suffices to consider the case where} $T$ has rank $r$, i.e.~$T=\sum_{j=1}^r x_j\otimes y_j\otimes z_j$. We have:
\[2T_S=\sum_{j=1}^r (y_j,x_j)^{\otimes 2}\otimes z_j+\sum_{j=1}^r (iy_j,-ix_j)^{\otimes 2}\otimes z_j.\]
{So if the} rank of $T$ is $r$ {then the} partially symmetric rank of $T_S$ is at most $2r$.

The claim follows by applying \autoref{cor:estimateparsym} to $T_S$.
\end{proof}
This implies that efficient, general criteria for proving nonsmoothability of simple algebras could lead to superlinear in $n$ lower bounds on {the border rank} of tensors. Further, efficient, general criteria {for proving} smoothability of simple algebras could lead to {upper bounds on the asymptotic rank} and the exponent \cite{kaski2024universal} of {any} tensor up to arbitrary precision. {This indicates that obtaining both types of criteria is extremely difficult. However, for further detail, we refer to \cite{Jel19}.}

The lower estimate $\frac{1}{2}k$ in  \autoref{cor:smoothabilitybrank1} makes
us lose a multiplicative constant $\frac{1}{2}$. This is due to the fact that
we insisted on {staying} in the realm of algebras and {had} to symmetrize our
tensor. {By working with $1_A$-generic tensors instead, we get results that are more efficient.}
Recall that $1_A$-generic tensors $T\in \bbK^{a}\otimes
\bbK^{b}\otimes
\bbK^{b}$ of minimal border rank $b$ correspond to smoothable modules and are
actively {studied}~\cites{landsberg2017abelian, wojtala,
jelisiejew2022components, JLP}.

\begin{prop}\label{prop:reductionto1gen}
For any concise tensor $T\in\bbK^a\otimes\bbK^b\otimes\bbK^c$ of border rank $k$ there exists a concise, $1_{\bbK^{a+1}}$-generic, minimal border rank $b+k$ tensor $T'\in \bbK^{a+1}\otimes\bbK^{b+k} \otimes \bbK^{b+k}$ such that $T$ is a restriction of $T'$. Explicitly, if $e_1,\dots, e_a$ is a basis of $\bbK^a$ and we extend it by $e_0$ to a basis of $\bbK^{a+1}$ then $T'=T+e_0\otimes i$, where $i$ is the identity on $\bbK^{b+k}$. As spaces of matrices:
\[
T'\bigl((\bbK^{b+k})^*\bigr)=\begin{pNiceMatrix}[vlines=4,margin]
x_1&\Cdots&x_b& y_1&\Cdots&y_{k-c}&z_1&\Cdots& z_c\\
\hline
\Block{3-3}{T\bigl((\bbK^c)^*\bigr)}&& &\Block{3-6}{\bigzero}\\
&&&&&\\
&&&&&
\end{pNiceMatrix},
\]
where the entries of $T\bigl((\bbK^c)^*\bigr)$ are linear forms in $z_1,\dots,z_c$.
\end{prop}
\begin{proof}
We proceed {by analogy with} the proof of \autoref{thm:twoestimate}. However, as we are no longer in the realm of algebras, we provide our rank and border rank estimates in a more direct way.

{If} $T_i\rightarrow T$, then $T_i'\rightarrow T'$. Thus, it is enough to
prove the statement under the assumption that $T$ has rank $k$. Let
$T=\sum_{j=1}^k a_j\otimes b_j\otimes c_j$. Let $\tilde T$ be the
corresponding generic tensor, i.e.~$\tilde T=\sum_{j=1}^k a_j\otimes
b_j\otimes \tilde c_j$, where $\tilde c_j$ form a basis of
$\bbK^{k}$. 
We divide the proof into two claims.

    {\textbf{Claim 1:}} $\tilde T'$ has border rank $b+k$. We will present the linear space $\tilde T'((\bbK^{b+k})^*)$ as the limit of spaces spanned by $(b+k)$-tuples of rank one matrices. For each integer $n$, the $b+k$ rank one matrices are as follows:
\begin{enumerate}[label=(\arabic*), left= 3pt, widest=2,nosep]
\item $b$ matrices $N_{j}$, for $j=1,\dots,b$, with {a non-zero} entry in the first row and column $j$.
\item $k$ matrices $M_{n,j}$, for $j=1,\dots, k$,
{where
\[M_{n,j}=\biggl(1,\frac{1}{n}a_j\biggr)\otimes \left((n\cdot b_j)+e_{b+j}\right),\] and $e_{b+j}$ is the $(b+j)$-th basis vector of $\bbK^{b+k}$.}
\end{enumerate}
{Clearly,} the matrices $N_j$ span the space corresponding to $\langle x_1,\dots,x_b\rangle$ in the statement, while {\[M_{n,j}-\sum_{s=1}^k n(b_j)_sN_s\] has a limit corresponding to $z_j$.}

    {\textbf{Claim 2:}} $\tilde T'$ restricts to $T'$. For every $e_i^*\in (\bbK^c)^*$ the layer $e_i^*(T)$ is a linear combination of $a_j\otimes b_j$, {so} of $e_j^*(\tilde T)$. We {can} use the same restrictions {for $\tilde T'$ so that we obtain the same tensor as $T'$ in first $b$-layers with respect to second tensor factor.}
    Then, it is {sufficient} to use an automorphism of $\bbK^k\subset\bbK^{b+k}$ to obtain exactly $T'$.

    The {above two claims} imply that $T'$ {does indeed have} border rank $b+k$.
\end{proof}
\begin{cor}\label{cor:brankto1Gen}
Fix a concise tensor $T\in\bbK^a\otimes\bbK^b\otimes\bbK^c$. Let $k$ be the smallest {integer, so that} the tensor $T'$ from the statement of  \autoref{prop:reductionto1gen} has border rank $b+k$. Then $T$ has border rank at least $k$ and at most $b+k$.
\end{cor}

\section{\normalfont \scshape\fontfamily{ptm}\selectfont Sweet pieces}

\noindent In this section we define a class of tensors that we call \emph{sweet pieces}. As we will argue, sweet pieces play an essential role when using  the Coppersmith-Winograd method to bound the constant $\omega$ that governs the complexity of matrix multiplication. Surprisingly, as we will see, sweet pieces have rank strictly smaller than the bounds used so far in the CW-method. The introduction of sweet pieces also allows to remove any reference to border rank and the Coppersmith-Winograd tensor to obtain upper bounds on $\omega$, as in \cite{CW}.
We start by recalling notions from complexity theory.  A good reference for
the methods we describe is \cite{BCS97}*{Part IV}.
We work over $\bbK$, which as before is algebraically closed and of
characteristic zero. Much of the literature below is done only for $\CC$. If necessary, we can always reduce to this case, by restricting $\bbK$ to a countably generated field (containing all
    coefficients of all involved tensors etc.) and embedding it into $\CC$.




{
\begin{defn}
Given three vector spaces $V_1$, $V_2$, and $V_3$,
for a tensor $T\in V_1\otimes V_2\otimes V_3$,
a \emph{blocking} is a triad of functions $f_i\colon\calB_i\to \bbZ^r$, where $\calB_i$ is a basis of $V_i$ for every $i=1,2,3$.
For any $(a_1,a_2,a_3)\in (\bbZ^r)^3$, a \emph{block} $T_{a_1,a_2,a_3}$
is the image of $T$ by the projection \[V_1\otimes V_2\otimes V_3\rightarrow
V_1'\otimes V_2'\otimes V_3',\] where $V_i'\subset V_i$ is defined as the linear subspace
\[
V_i'\coloneqq \langle \Set{b\in\calB_i|f_i(b)=a_i}\rangle.
\]
A block is said to be in the \emph{support} of the tensor, if it is nonzero.
A tensor with a blocking $(f_1,f_2,f_3)$ is said to be \emph{tight} if for every block $(a_1,a_2,a_3)$ in the support, $a_1+a_2+a_3=0\in \bbZ^r$.
\end{defn}
Sometimes we will write \textit{block $(a_1,a_2,a_3)$}, instead of $T_{a_1,a_2,a_3}$.
\begin{rem}
In the literature one either refers to:
\begin{itemize}
\item tight tensors, assuming one can choose bases of each $V_i$ and the blocking is given by injective functions, i.e., the blocks are $1\times 1 \times 1$ tensors \cite{conner2021towards}, or
\item tight sets; in our case, this corresponds to fixing bases of the vector spaces $V_1,V_2,V_3$ \cite[Section 15.7]{Algebraic_Complexity_Theory}.
\end{itemize}
The definition above is flexible enough to easily interpolate between these two.
\end{rem}
By a (probability) distribution $p$ on any set $S$ we mean a function $p:S\rightarrow \mathbb{R}_{\geq 0}$, such that \[
\sum_{s\in S} p(s)=1.
\]
Thus, a probability distribution $p$ on the support $S$ of a tensor $T$ with fixed blocking assigns a non-negative real number $p(a,b,c)$ to each $(a,b,c)$ indexing a block in the support of $T$. Such a distribution induces three marginal distributions $p_1,p_2,p_3$, where $p_i$ is a distribution on the image of $f_i$. More precisely, \[p_1(a)\coloneqq\sum_{(a,b,c)\in S}p(a,b,c)\] and similarly for $p_2$ and $p_3$.
}

The following definition has appeared implicitly in many articles on fast matrix multiplication. While it looks technical, in practice it means that the tensor is sufficiently symmetric so that it shares many properties with a tensor that is simply a direct sum of copies of the same tensor.
\begin{defn}\label{def:sweet}
A \emph{sweet piece} is a tight tensor $T$, with a fixed blocking $f_1,f_2,f_3$ such that:
\begin{enumerate}[label=(\arabic*), left= 3pt, widest=2,nosep]
\item each block in the support is isomorphic to a fixed tensor $U$;
\item the uniform distribution on support has three equal marginals, which are also uniform.
\end{enumerate}
The cardinality of the image of each $f_i$, which by the definition does not depend on $i$, is denoted by $p_T$.
\end{defn}

\begin{defn}
For a tensor $T$ its \emph{asymptotic rank} is defined by \[
\lim_{n\rightarrow\infty} R(T^{\boxtimes n})^{\frac{1}{n}},
\]
where $R(T^{\boxtimes n})$ is the rank of the $n$-th Kronecker power of $T$.
For a concise tensor $T\in \bbK^a\otimes \bbK^b\otimes \bbK^c$ the asymptotic rank is always greater than or equal to $\max(a,b,c)$ and we say that a tensor has \emph{minimal asymptotic rank} if equality holds.
\end{defn}
For example, the $2\times 2$ matrix multiplication tensor $M_2$ has asymptotic rank $2^{\omega}$.
The following theorem can be considered as the cornerstone of the Coppersmith-Winograd method.
\begin{teo}\label{thm:sweet_bound}
    Let $T\in (\bbK^{p_Ta^2})^{\otimes 3}$ be a sweet piece of
asymptotic rank $r$, with blocking $(f_1,f_2,f_3)$,  where each block $U$ is isomorphic to the matrix multiplication tensor $M_a$, with $a>1$. Then
\[\omega \leq \log_a (r/p_T)= \log_a (r/|\im f_i|).\]
In particular, if $T$ has minimal
 asymptotic rank, then $\omega=2$.
\end{teo}
\begin{proof}
Consider the uniform probability distribution for all blocks and apply the main Coppersmith-Winograd theorem, for example as in \cite{BCS97}*{Theorem 15.41}.
\end{proof}
The proof above suggests that \autoref{thm:sweet_bound} is a simple consequence of the results of Coppersmith and Winogrand. However, we want to emphasize  that sweet pieces play an absolutely central role in the whole theory. In particular, the main part of \cite{CW} can be restated as the following lemma.
\begin{lem}
\label{lem:Cop-Win}
Let $T$ be a sweet piece and $\varepsilon >0$. There exists $N\gg 0$, such that $T^{\boxtimes N}$ restricts to $\lfloor(p_T-\varepsilon)\rfloor^N$ disjoint copies of $U^{\boxtimes N}$.
\end{lem}
We emphasize that in \autoref{lem:Cop-Win} we do not need to degenerate the tensor to get a direct sum \[(U^{\boxtimes N})^{\oplus \lfloor(p_T-\varepsilon)\rfloor^N},\] but only to restrict it\footnote{However, the proof that such a restriction exists is not constructive and relies on probabilistic arguments.}.

The CW-method also provides a way to create new sweet pieces. One starts with a tight tensor $T$, with blocking $f_1,f_2,f_3$, and a probability distribution $P$ on the support of $T$. Note that the blocking on $T$ induces the blocking on $T^{\boxtimes N}$, where the indices of the blocks are sequences of length $N$ of indices for blocks of $T$. In particular, each block $B$ in $T^{\boxtimes N}$ provides three sequences of length $N$. The $i$-th sequence gives a probability distribution $p_i(B)$ on the image of $f_i$, where the probability of $a\in \im f_i$ is equal to the number of times $a$ appears in the $i$-th sequence divided by $N$.

\begin{prop}\label{prop:SP}
Let $T$ be a tight tensor with a fixed blocking.
Assume $P$ is a probability distribution on blocks $b$ in the support of $T$ that is reconstructable from its three marginals\footnote{i.e.~it is the unique probability distribution on the support with given marginals} and all three marginals are the same distributions. Assume that $N$ is such an integer that $Np$ is integral.
Then, the projection of $T^{\boxtimes N}$ onto the basis vectors belonging to the blocks $B$, such that for $i=1,2,3$ the $i$-th marginal of $P$ is $p_i(B)$, is a sweet piece.
\end{prop}
\begin{proof}
Since the three marginals for $T$ are the same, by the action of the permutation
group $\Sigma_N$ on the entries of the indexing sequences and the action of
$\Sigma_3$ on the sequences themselves, the second point in  \autoref{def:sweet} holds. We have to prove that all blocks in the support are isomorphic. Indeed, we claim that each one is isomorphic to the Kronecker product of $N$ blocks of $T$, where each block $b$ appears $P(b)N$ many times.
Fix a block $B$ in $T^{\boxtimes N}$, indexed by the three sequences \[ (a_1,\dots,a_N),\quad (a_1',\dots,a_N'),\quad (a_1'',\dots,a_N'').\] Then $B$ is isomorphic to the Kronecker product of $T_{a_i,a_i',a_i''}$ i.e.~blocks of $T$ indexed by $a_i,a_i',a_i''$.
The three sequences provide a probability distribution $P'$ on the support blocks of $T$. By the choice of the projection, $P'$ and $P$ have the same marginals. By the assumption on $P$, we have $P=P'$ and the claim follows.
\end{proof}
The sweet piece of \autoref{prop:SP} is denoted by $SP(T)$ or $SP_{P,N}(T)$. Even if $T$ is not tight, as we will see, the same construction, via projection of Kronecker powers, may lead to sweet pieces.
Let $T\in V_1\otimes V_2\otimes V_3$ be a (possibly not tight) tensor with a
fixed blocking. Fix a $\bbK^*$ action on each $V_i$, where the fixed basis vectors are eigenvectors and the weight $w(e)=w\bigl(f_i(e)\bigr)$ of the action on a basis vector $e$ depends only on $f_i(e)$.
We obtain induced weights of blocks, where the block indexed by
$(a_1,a_2,a_3)$ has weight \[w(a_1)+w(a_2)+w(a_3).\] Assume that for non-zero
blocks the weights are non-negative. Let $T'$ be the toric degeneration of $T$,
with the parameter $\bbK^*\ni t\rightarrow 0$.
\begin{lem}\label{lem:samesweet}
Using the notation introduced above, assume that the toric degeneration $T'$ satisfies the assumptions of \autoref{prop:SP}. If $SP(T')\neq 0$, then $SP(T)=SP(T')$.
\end{lem}
\begin{proof}
Clearly, $SP(T)$ and $SP(T')$ agree on blocks where $SP(T')$ is not zero. So it remains to show that if a block is zero in $SP(T')$, then it is also zero in $SP(T)$. Pick a zero block in $SP(T')$ indexed by three sequences. By contradiction, assume that this block is non-zero in $SP(T)$. Consider the probability distribution $P$ on the blocks of $T$ induced by the three sequences. As the block was nonzero in $SP(T)$ and zero in $SP(T')$, the probability distribution must be supported on blocks of non-negative weight in $T$ and positive on at least one block of positive weight. Let $P'$ be the probability distribution on the blocks of $T$ induced by the three sequences labeling a block in the support of $SP(T')$. Note that $P$ and $P'$ have the same marginals. Thus, $P-P'$ is negative only on blocks with weight zero, is positive on at least one block with positive weight and has all three marginals equal to zero.

Let $S$ be the sum of the weights of all the blocks, each one multiplied by the value $P-P'$ attains on that block. Looking at the marginals, we see that $S=0$. However, by summing over all the blocks we see that $S>0$, which is a contradiction and concludes the proof.
\end{proof}
\begin{defn}
    Let $G$ be a finite abelian group.
We identify $G$ with the basis of $\bbK^G$.
We define $T_G\in \bbK^G\otimes\bbK^G\otimes\bbK^G$ in this basis by
\[T_G:=\sum_{g_1+g_2=g_3\in G} g_1\otimes g_2\otimes g_3.\]
By the discrete Fourier transform this is a tensor of rank $|G|$ and thus is isomorphic to the unit tensor.
The tensor $CW_n\in \bbK^n\otimes\bbK^n\otimes\bbK^n$ is defined as
\[CW_n\coloneq \sum_{i=1}^n e_1\otimes e_i\otimes e_i+\sum_{i=2}^n e_i\otimes e_1\otimes e_i+\sum_{i=2}^{n-1} e_i\otimes e_i\otimes e_n.\]
\end{defn}
The next lemma and corollary can be seen as variants or special cases of \cite{alman2018further} and \cite{hoyois2021hermitian}*{Proposition~4.1}. The first paper explicitly states that $T_G$ can replace $CW_{|G|}$ for the bounds on $\omega$. In the second, the authors provide degenerations of a larger class of tensors to $CW_n$. The main difference with the second article is that we can identify sweet pieces of two tensors. Unlike the first article, we can later apply our results to obtain new rank estimates of the sweet pieces of both tensors.
\begin{lem}\label{lem:sweetCWnobrank}
For any abelian group $G$ of cardinality at least three, the structure tensor $T_G$ has a $\bbK^*$ degeneration into the big Coppersmith-Winograd tensor $CW_{|G|}$, by using the blocking as in \cite{CW}.
\end{lem}
\begin{proof}
Let $g\in G$ be a non-neutral element. Assign weight $0$ to the neutral element, weight $1$ to all elements different from $g$ and the neutral element, and weight $2$ to $g$ for the first two copies of $\bbK^G$ and minus these weights for the third copy. The toric degeneration is isomorphic to $CW_{|G|}$.
\end{proof}
\begin{cor}\label{cor:sweetCW=G}
The tensors $T_G$ and $CW_{|G|}$ have the same sweet pieces. In particular, $SP_N(CW_n)$ has \emph{rank} at most $n^N$.
\end{cor}
The above corollary allows us to get better bounds for the rank of $SP(CW_n)$.
\begin{defn}
Let $T$ be a tensor with a blocking and a probability distribution $P$ on the blocks. A \emph{chimney} is the image of $T^{\boxtimes N}$ under the projection map: \[V_1^{\otimes N}\otimes V_2^{\otimes N}\otimes V_3^{\otimes N}\rightarrow V_1'\otimes V_2'\otimes V_3^{\otimes N},\] where, for $i=1,2$, the subspace $V_i'$ is spanned only by basis vectors indexed by sequences that produce the same probability distribution as the $i$-th marginal of $P$.
By changing $i=1,2$ to $i=1,3$ or $i=2,3$ we get two other similar chimneys.
\end{defn}
Note that if we think of $T^{\boxtimes N}$ as a three-dimensional cuboid, we should think of the sweet piece as the intersection of three chimneys.
\begin{prop}\label{prop:sweetRank}
Consider a probability distribution on the six blocks in the support of $CW_n$ that assigns $p$ to each of the three large blocks of format $(n-2)^2\times 1$ and $q=1/3-p$ to each of the three small blocks of format $1^3$.
The rank of $SP_{N,P}(CW_n)$ is smaller than \[n^N-\binom{N}{\frac{2}{3}N+1}(n-1)^{\frac{1}{3}N-1}.\] In particular, it is smaller than border rank of $(CW_n)^{\boxtimes N}$.
\end{prop}
\begin{proof}
By \autoref{cor:sweetCW=G} it is enough to prove the statement for $SP_{N,P}(T_G)$ for $|G|=n$. Note that $(T_G)^{\boxtimes N}$ is a tensor of minimal rank $n^N$.

Let us consider the corresponding chimney with entries indexed by triples of sequences of group elements, where for the first two sequences the neutral element appears $(p+2q)N$ times and the distinguished group element $g\in G$ appears $qN$ times. Note that the third sequence is arbitrary, as we are working with the chimney, not the sweet piece.

Note also that an entry in the chimney indexed by $(g_i), (g_i'), (g_i'')$ is nonzero only if:
for any $i=1,\dots, N$ we have the implication:
if $g_i=e$ then $g_i'=g_i''$. In particular, as we have $(p+2q)N$ neutral elements in the first sequence and only $qN$ elements equal to $g$ in the second sequence, the entries are always zero if the third sequence contains less than $(p+q)N$ elements distinct from $g$. In particular, there are at least \[\binom{N}{(2p+2q)N+1}(n-1)^{(p+q)N-1}=\binom{N}{\frac{2}{3}N+1}(n-1)^{\frac{1}{3}N-1}\] sequences indexing layers of the chimney that have only zero entries.

We now apply the classical substitution method (\cite{landsberg2017abelian}*{\S 3},  \cite{pan1966means}) for the tensor $(T_G)^{\boxtimes N}$ and the layers that are zero in the chimney. Note that we do not change the sweet piece, since the layers, that are nonzero in the whole tensor, are zero over the sweet piece. The result follows.
\end{proof}
\begin{rem}
Clearly, the previous estimate can be made much better. We believe that estimating the ranks and the border ranks of sweet pieces may lead to new discoveries about the bounds on $\omega$. We also emphasize that if we bound the ranks of sweet pieces, all the barrier results, as in  \cite{CVZ21}
do not apply. In particular, it is in principle possible to prove $\omega=2$, even using arbitrarily large Coppersmith-Winograd tensors, provided we find better bounds on the ranks of sweet pieces.
\end{rem}

Our next aim is to explain relations between sweet pieces and smoothability of algebras.
We note that $CW_n$ is the structure tensor of the algebra $A$ apolar to the quadric \[f\coloneqq x_1^2+\dots+x_{n-2}^2.\]
Hence, $CW_n^{\boxtimes N}$ is the structure tensor of the algebra $A^{\otimes N}$ apolar to $f^{\boxtimes N}$. We note that $A$ comes with the standard grading and has Hilbert function $(1,n-2,1)$. Thus $A^{\otimes N}$ is multigraded and with respect to the total grading \[\dim_{\bbK} (A^{\otimes N})_i=\sum_{j=0}^{\lfloor \frac{i}{2}\rfloor} \binom{N}{j}\binom{N-j}{i-2j}(n-2)^{i-2j}.\]
The grading on $A$ corresponds exactly to the blocking of $CW_n$ introduced in \cite{CW}. The total grading on $CW_n^{\boxtimes N}$ corresponds to the blocking on Kronecker powers of this tensor, which has been used most often so far.

\begin{lem}
Let $T$ be a structure tensor of a graded algebra $A$, with the blocking induced by the grading. Consider the total grading on $A^{\otimes N}$. Let $P$ be a probability distribution on the blocks of $T$ that is reconstructible from its marginals $P_1$, $P_2$, $P_3$ that have the same distribution. Let $N$ be such an integer that $N\cdot P$ is integral.
Then, there exists a degree $i$, such that $SP_{P,N}(T)$ is a restriction of the tensor representing the map:
\[(A^{\otimes N})_i\times (A^{\otimes N})_i\rightarrow (A^{\otimes N})_{2i}.\]
\end{lem}
\begin{proof}
An indexing sequence of $N$ elements of $A$ is an indexing sequence for the first coordinate in $SP(T)$ if and only if elements of degree $d$ appear $P_1(d)N$ times in it. Thus the total degree of such a sequence is $i:=\sum_d P_1(d)dN$.
\end{proof}
In the setting of the previous lemma
$SP(T)$ may also be seen as a restriction of the tensor that is the structure tensor of the subalgebra of $A^{\otimes N}$ generated in degree $i$. Such algebras are in most cases of much smaller dimension, even asymptotically with $N$.
\begin{exam}\label{exam:Vero}
Let $A=\bbK [x,y]/(x^2,y^2)$ be the algebra representing the $CW_4$ tensor. Let $P$ be the probability distribution equal to $1/3$ on the large blocks of the support of format $2^2\times 1$. The Veronese subalgebra of $A^{\otimes 3k}$ generated in degree $k$ has dimension \[2+2\sum_{a=0}^k \binom{3k}{a}\cdot\binom{3k-a}{2k-a}.\] The sweet piece $SP_{P,3k}(CW_4)$ is a subtensor of the tensor associated to that Veronese algebra.
\end{exam}
One way to realize the Veronese subalgebra of $A^{\otimes N}$ generated in degree $i$ is to fix a $G_i:=\bbZ/i\bbZ$ action on $A$, where elements of degree one are acted upon with $1\in G_i$. We can also think of it as multiplication by the complex $i$-th root of unity. Then the Veronese subalgebra we are interested in is given by the invariants $(A^{\otimes N})^{G_i}$. Unfortunately, as we have seen in \autoref{exm:VerNotSmooth}, even if $A$ is smoothable and so is $ A^{\otimes N}$, the algebra of invariants does not have to be.

Let $T\in (\bbK^{n})^{\otimes 3}$ be a tensor giving rise to sweet piece $SP(T)$, where each block is a matrix multiplication tensor.
Note that the only asymptotically non-optimal estimation in the Coppersmith-Winograd method is precisely the point where the size of $T^{\otimes N}$, that is $n^N$, is usually much larger than the size of $SP(T)$.
In particular, either
\begin{itemize}
\item if $SP(T)$ grows in size as $n^N$, as it is the case for the symmetric
    tensor $T\in (\bbK^3)^{\otimes 3}$ given by the monomial $xyz$, and $T$ is of minimal asymptotic rank (cf.~\cite[Remark 15.44]{BCS97}), or
\item if one can prove that $SP(T)$ has minimal asymptotic rank,
\end{itemize}
then $\omega=2$, by \autoref{thm:sweet_bound}. Further,
asymptotically in $N$ improvements on the (asymptotic) border rank of $(A^{\otimes N})^{G_i}$ would lead to improvements on the bounds on $\omega$.

In the last part of the article we exhibit relations of sweet pieces to NP-hard problems in complexity theory.
Consider the algebra $B:=\bbK[x]/(x^2)$ and let $T_B\in (\bbK^2)^{\otimes 3}$ be the associated tensor. It is natural to identify a basis of $\bbK^2$ with subsets of the set with one element.
Thus the tensor $T_B$ has the three entries equal to one exactly in $3$-way partitions of the one element set \cite{bjorklund2023asymptotic}.
We note that the algebra $A$ of \autoref{exam:Vero} is $B^{\otimes 2}$.
Consider the tensor \[T_N:=T_B^{\boxtimes N},\] where we now identify the basis of $\bbK^{2^N}$ with subsets of the set with $N$ elements. Note that $T_N$ encodes $3$-way partitions of the
set with $N$ elements.
Let us fix $N:=3k$ and the uniform probability distribution on the three entries of $T_B$. We obtain the sweet piece $SP_N{T_B}$ that is exactly equal to the tensor $T_N$ from \cite{pratt2023stronger}.
This sweet piece is also the tensor representing the degree one multiplication in the $k$-th Veronese subalgebra of $B^{\otimes N}$. The dimension of this Veronese subalgebra equals twice
the size of the sweet piece plus two. Combining this observation with \cite{pratt2023stronger}*{Corollary 1.12} we obtain the following corollary.
\begin{cor}
If the asymptotic rank of a sweet piece coming from tensors with minimal border rank has minimal asymptotic rank or if Veronese subalgebras of smoothable algebras give rise to tensors of minimal
asymptotic rank, then the set covering conjecture from \cite{cygan2016problems} is false.
\end{cor}
We note that by \cite{pratt2023stronger}*{Corollary 1.12} one does not need exponential improvements in the border rank of the tensors $T_N$ to disprove the set covering conjecture. The obvious bound
is $8^k$, while the current best known bound is $\frac{1}{2}8^k$ \cite{pratt2023stronger}*{p.~4, point (4)}. In fact, it would be enough to prove that asymptotic rank is at most
\[\frac{2\cdot 8^k}{9\cdot k}.\] Below we provide a modest improvement on the upper bound.
\begin{prop}\label{prop:Prattimproved}
The rank of the tensor $T_N$ is at most \[\frac{1}{2}8^{k}-\sum_{i=k+1}^{\lfloor N/2\rfloor} \binom{N}{2i}.\]
\end{prop}
\begin{proof}
We proceed similarly to \autoref{prop:sweetRank}. We start with $T_{\bbZ_2}\in
(\bbK^2)^{\otimes 3}$ and note that $T_N$ is a subtensor of the minimal rank tensor $T_{\bbZ_2}^{\boxtimes N}$.
By the substitution method the rank of $T_N$ is bounded by the number of sets obtained as the symmetric difference of two subsets of cardinality $k$ of a set of cardinality $N$.
We note that such subsets are always even and of cardinality at most $2k$. The bound follows.
\end{proof}

\section*{\normalfont \scshape\fontfamily{ptm}\selectfont Acknowledgements}
\noindent
The authors would like to thank Alessandra Bernardi, Maciej Ga{\l}{\k{a}}zka,
Fulvio Gesmundo, Hang Huang, Joseph
M.~Landsberg, Giorgio Ottaviani, Kaski Petteri, Daniele Taufer,
and Virginia Vassilevska Williams for very helpful comments. The authors would like to thank the anonymous referee for careful reading.
This work is partially supported by  the Thematic Research Programme
\textit{Tensors: geometry, complexity and quantum entanglement}, University of
Warsaw, Excellence Initiative – Research University and the Simons Foundation
Award No.~663281 granted to the Institute of Mathematics of the Polish Academy
of Sciences for the years 2021-2023.
A special role in the realization of this paper was taken by the semester program AGATES: \textit{Algebraic Geometry with Applications to Tensors and Secants}, held in Warsaw from September 12 to December 16, 2022. This paper was written while the first author was a research fellow at \textit{Università degli Studi di Firenze}. The first author has been supported by the scientific project \textit{Multilinear Algebraic Geometry} of the program \textit{Progetti di ricerca di Rilevante Interesse Nazionale} (PRIN), Grant Assignment Decree No.~973, adopted on 06/30/2023 by the Italian Ministry of University and Research (MUR) and by the project \textit{Thematic Research Programmes}, Action I.1.5 of the program \textit{Excellence Initiative -- Research University} (IDUB) of the Polish Ministry of Science and Higher Education.
The second author is supported by National Science Centre grant
2020/39/D/ST1/00132. The third author is supported by the Deutsche Forschungsgemeinschaft grant
467575307.

\end{document}